\documentclass[10pt]{elsarticle}

\usepackage{calrsfs,amsmath,amssymb,graphicx,array,accents,natbib,figlatex,epsfig}
\usepackage{setspace,wrapfig,marvosym, bbm, stmaryrd,centernot, amsthm, color,upgreek,arydshln}

\newtheorem{theorem}{Theorem}[section]
\newtheorem{lemma}[theorem]{Lemma}
\newtheorem{defn}{Definition}
\newtheorem{rem}{Remark}[section]

\newcommand{\R}{\mathbb{R}}

\newlength{\kaka}

\newcommand{\ahref}[2]{}

\newcommand{\obs}{^{\text{obs}}}

\newcommand{\ff}{^{\tiny \iota}}

\newcommand{\beq}{\begin{equation}}
\newcommand{\eeq}{\end{equation}}
\newcommand{\lb}{\label}

\newcommand{\bea}{\begin{eqnarray}}
\newcommand{\eea}{\end{eqnarray}}
\newcommand{\bxr}{\begin{array}}
\newcommand{\exr}{\end{array}}

\newcommand\exs{\hspace*{0.4mm}}
\newcommand\xxs{\hspace*{0.2mm}}

\newcommand\nxs{\hspace*{-0.2mm}}
\newcommand\txs{\hspace*{-0.1mm}}

\newcommand{\norms}[1]{\parallel\! #1 \!\parallel} 
\newcommand{\dbr}[1]{\llbracket #1 \rrbracket}

\newcommand{\bSig} {\boldsymbol{\Sigma}}

\newcommand{\bC} {\boldsymbol{C}}
\newcommand{\bK} {\boldsymbol{K}}

\newcommand{\bR} {\boldsymbol{\sf R}}

\newcommand{\bn} {\boldsymbol{n}}
\newcommand{\ba} {\boldsymbol{a}}

\newcommand{\bx} {\boldsymbol{x}}
\newcommand{\by} {\boldsymbol{y}}
\newcommand{\be} {\boldsymbol{e}}

\newcommand{\bz} {\boldsymbol{z}}
\newcommand{\bg} {{\boldsymbol{g}}}
\newcommand{\bq} {{\boldsymbol{q}}}

\newcommand{\butt}{{\texttt{\bf u}}}
\newcommand{\ptt}{{\texttt{p}}}

\newcommand{\bd} {\boldsymbol{d}}
\newcommand{\bI} {\boldsymbol{I}}

\newcommand{\pff}{\boldsymbol{u}^{\textit{i}}}

\newcommand{\bUsfr}{{{\textrm{\bf U}}^{\mathfrak{s}}}}
\newcommand{\buffr}{{{\textrm{\bf u}}^{\mathfrak{f}}}}
\newcommand{\Usfr}{{{\textrm{U}}^{\mathfrak{s}}}}
\newcommand{\uffr}{{{\textrm{u}}^{\mathfrak{f}}}}
\newcommand{\bpsfr}{{{\textrm{\bf p}}^{\mathfrak{s}}}}
\newcommand{\psfr}{{{\textrm{p}}^{\mathfrak{s}}}}
\newcommand{\pffr}{{{\textrm{p}}^{\mathfrak{f}}}}
 
\newcommand{\bTsfr}{{{\textrm{\bf T}}^{\mathfrak{s}}}}
\newcommand{\btffr}{{{\textrm{\bf t}}^{\mathfrak{f}}}}
\newcommand{\Tsfr}{{{\textrm{T}}^{\mathfrak{s}}}}
\newcommand{\tffr}{{{\textrm{t}}^{\mathfrak{f}}}}
\newcommand{\bqsfr}{{{\textrm{\bf q}}^{\mathfrak{s}}}}
\newcommand{\qsfr}{{{\textrm{q}}^{\mathfrak{s}}}}
\newcommand{\qffr}{{{\textrm{q}}^{\mathfrak{f}}}}

\newcommand{\sip} {\!\cdot\!}

\newcommand{\OOd}{\Omega_{\bd}}

\newcommand{\bzero}{\boldsymbol{0}}

\newcommand{\bu} {\boldsymbol{u}}

\newcommand{\bt} {{\boldsymbol{t}}}

\newcommand{\bxi} {\boldsymbol{\xi}}

\newcommand{\btu} {\text{\bf{u}}}

\newcommand{\bPhi}{\boldsymbol{\Phi}}

\newcommand{\dualGA}[2]{\left< #1, #2 \right>_{\Gamma}}
\newcommand{\bphi}{{\boldsymbol{\varphi}}}

\makeatletter
\newsavebox{\@brx}
\newcommand{\llangle}[1][]{\savebox{\@brx}{\(\m@th{#1\langle}\)}%
  \mathopen{\copy\@brx\kern-0.65\wd\@brx\usebox{\@brx}}}
\newcommand{\rrangle}[1][]{\savebox{\@brx}{\(\m@th{#1\rangle}\)}%
  \mathclose{\copy\@brx\kern-0.65\wd\@brx\usebox{\@brx}}}
\makeatother

\newcommand{\bpsi}{{\boldsymbol{\psi}}}
\begin{document}

\begin{frontmatter}

\title{\textcolor{black}{Poroelastic near-field inverse scattering}}

\author{Fatemeh Pourahmadian$^{1,2}$\corref{cor1}} 
\author{Kevish Napal$^1$}
\address{$^1$ Department of Civil, Environmental \& Architectural Engineering, University of Colorado Boulder, USA}
\address{$^2$ Department of Applied Mathematics, University of Colorado Boulder, USA}
\cortext[cor1]{Corresponding author: tel. 303-492-2027, email {\tt fatemeh.pourahmadian@colorado.edu}}

\date{\today}

\begin{abstract}
A multiphysics data analytic platform is established for imaging poroelastic interfaces of finite permeability (e.g., hydraulic fractures) from elastic waveforms and/or acoustic pore pressure measurements. This is accomplished through recent advances in design of non-iterative sampling methods to inverse scattering. The direct problem is formulated via the Biot equations in the frequency domain where a network of discontinuities is illuminated by a set of total body forces and fluid volumetric sources, while monitoring the induced (acoustic and elastic) scattered waves in an arbitrary near-field configuration. A thin-layer approximation is deployed to capture the rough and multiphase nature of interfaces whose spatially varying hydro-mechanical properties are a priori unknown. In this setting, the well-posedness analysis of the forward problem yields the admissibility conditions for the contact parameters. In light of which, the poroelastic scattering operator and its first and second factorizations are introduced and their mathematical properties are carefully examined. It is shown that the non-selfadjoint nature of the Biot system leads to an intrinsically asymmetric factorization of the scattering operator which may be symmetrized at certain limits. These results furnish a robust framework for systematic design of regularized and convex cost functionals whose minimizers underpin the multiphysics imaging indicators. The proposed solution is synthetically implemented with application to spatiotemporal reconstruction of hydraulic fracture networks via deep-well measurements.        
\end{abstract}

 
 \end{frontmatter}

\section{Introduction} \label{sec1}

Coupled physics processes driven by engineered stimulation of the subsurface underlie many emerging technologies germane to advanced energy infrastructure such as unconventional hydrocarbon and geothermal reservoirs~\cite{hofm2019,cau2017}. Optimal design and closed-loop control of such systems require real-time feedback on the nature of progressive variations in a target region~\cite{mass2017,Maxw2014}. 3D in-situ tracking of multiphasic developments however is exceptionally challenging~\cite{hou2021,snee2020}. Engineered treatments (such as fluid or gas injection) often occur in a complex background whose structure and hydro-mechanical properties are uncertain~\cite{zoba2019,Geoe2017}. Nevertheless, state-of-the-art solutions to in-situ characterization mostly rely on elastic waves~\cite{wape2021,sun2020,meti2018,Maxw2014}. In particular, micro-seismic waves generated during shear propagation of fractures are commonly used to monitor evolving discontinuities in rock~\cite{verd2020,Resh2010}. Such passive schemes are nonetheless agnostic to aseismic evolution~\cite{calo2011}, and involve significant assumptions on the nature of wave motion in the subsurface which may lead to critical errors~\cite{Gaje2017,shap2015,baig2010}. On the other hand, existing approaches to full waveform inversion~\cite{teod2021,trin2019,meti2018}, while offering more comprehensive solutions, are by and large computationally expensive and thus inapplicable for real-time sensing. Therefore, there exists a critical need for the next generation of imaging technologies that transcend some of these limitations. 

In this vein, recent advances in rigorous design of sampling-based waveform tomography solutions~\cite{Audibert2014,audi2017,bonn2019,pour2019,cako2016} seem particularly relevant owing to their robustness and newfound capabilities pertinent to imaging in complex and unknown domains~\cite{audi2017,pour2019,pour2020,pour2020(2)}. So far, these developments mostly reside in the context of self-adjoint and energy preserving systems such as Helmholtz and Navier equations. Wave attenuation, however, is a defining feature of the fractured rock systems with crucial impact on subsurface exploration~\cite{rubi2014,germ2013,alla2009}. In this context, a systematic analysis of sound absorption -- e.g.,~due to the fluid-solid interaction, and its implications on sampling-based data inversion is still lacking. This, in part, may be due to the sparse and single-physics nature of traditional data. Fast-paced progress in sensing instruments such as fiber optic (FBG) sensors -- capturing high-resolution multiple physical measurements in time-space~\cite{verd2020, qiao2017}, may bridge this gap and enable a holistic approach to subsurface characterization. This invites new data analytic tools capable of (a) expedited processing of large datasets, and (b) simultaneous inversion of multiphysics observations. The latter may particularly assist a high-fidelity reconstruction of the sought-for subterranean events.  

In this paper, the theoretical foundation of poroelastic inverse scattering is established with application to spatiotemporal tracking of hydraulic fracture networks in~unconventional energy systems. This method is nested within the framework of active sensing where the process zone is sequentially illuminated (during stimulation) via deep-well excitations prompting poroelastic scattering whose signature is recorded in the form of seismic and pore pressure waveforms. Thus-obtained multiphase sensory data are deployed for geometric reconstruction of treatment-induced evolution in the target region. Rooted in rigorous foundations of the inverse scattering theory~\cite{cako2016}, the proposed imaging solution carries a high spatial resolution with carefully controlled sensitivity to noise and illumination frequency.

The forward problem is posed in a suitable dimensional platform catering for simultaneous elastic and acoustic data inversion. A system of arbitrary-shaped hydraulic fractures with non-trivial (i.e.,~heterogeneous, dissipative, and finitely permeable) interfaces is illuminated by a plausible combination of total body forces and fluid volumetric sources in a medium governed by the coupled Biot equations~\cite{biot1956,biot1962}. Thus-induced multiphase scattered fields are used to define the poroelastic scattering operator mapping the source densities to near-field measurements. In this setting, the reconstruction scheme is based on (i) a custom factorization of the near-field operator, and (b) a sequence of approximate solutions to the scattering equation, seeking (fluid and total-force) source densities whose affiliated waveform pattern matches that of a designated poroelastodynamics solution emanating from a sampling point. The latter is obtained via a sequence of carefully constructed cost functionals whose minimizers can be computed without iterations. Such minimizing solutions are then used to construct a robust multiphysics imaging indicator whose performance is examined through a set of numerical experiments. 

In what follows, Section~\ref{PS} introduces the direct scattering problem and the necessary tools for the inverse analysis. In particular, Section~\ref{Prelim} investigates the well-posedness of the forward problem which leads to finding the admissibility conditions for the interfacial contact parameters including the elastic stiffness matrix and the permeability, effective stress and Skempton coefficients. This criteria plays a key role later in Section~\ref{SFS} where the essential properties of the poroelastic scattering operator and its components, in the first and second factorizations, are established. Section~\ref{SSA} presents the regularized cost functionals for solving the scattering equations, main theorems, and imaging functionals. Section~\ref{numerics} illustrates a set of numerical experiments implementing the reconstruction of (stationary and evolving) hydraulic fracture networks via the near-field measurements by way of the proposed imaging indicators.

\pagebreak
\section{Problem statement}\label{PS}

Consider the ball $\mathcal{B}'$ of radius $R\!>\!0$, centered at the origin in the subsurface $\Omega \subset \R^3$, encompassing the stimulation zone and sensing grid $\mathcal{G}$ as illustrated in Fig.~\ref{fig1}. Let $R$ be sufficiently large so that the probing waves may be assumed negligible in $\Omega \setminus \mathcal{B}'$ due to poroelastic attenuation. It is also assumed that $\mathcal{B}'$ does not meet $\partial \Omega$, which in hydraulic fracturing is a reasonable premise, in particular, since a treatment region is typically four to ten kilometers below the surface while $R = O(10^2)$m~\cite{hyma2016}. The domain $\mathcal{B}' \subset \Omega$ is described by drained Lam\'{e} coefficients $\lambda'$ and $\mu'$, Biot modulus $M'$, total density $\rho'$, fluid density $\rho_f'$, apparent mass density $\rho_a'$, permeability coefficient $\kappa'$, porosity $\phi$, and Biot effective stress coefficient $\alpha$. A system of pre-existing and stimulation-induced fractures and flow networks $\Gamma'$ is embedded in $\mathcal{B}'$ whose rough and multiphasic interfaces are characterized by a set of heterogeneous  parameters, namely: the stiffness matrix~$\bK'(\bxi')$, permeability coefficient~$\varkappa_{\textrm{\tiny f}}'(\bxi')$, effective stress coefficient~$\alpha_{\textrm{\tiny f}}(\bxi')$, Skempton coefficient~$\beta_{\textrm{\tiny f}}(\bxi')$, and fluid pressure dissipation factor $\Pi(\bxi')$ on $\bxi' \in \Gamma'$. The domain is excited by time harmonic total body forces of density $\bg'_s(\bxi',\omega')$ and fluid volumetric sources of density $g'_f(\bxi',\omega')$ on $\bxi' \in \mathcal{G}$ at frequency $\omega'$. This gives rise to poroelastic scattering observed on the sensing grid $\mathcal{G}$ in terms of solid displacements $\bu'(\bxi',\omega')$ and pore pressure $p'(\bxi',\omega')$.   

\begin{figure}[tp]
\center\includegraphics[width=0.7\linewidth]{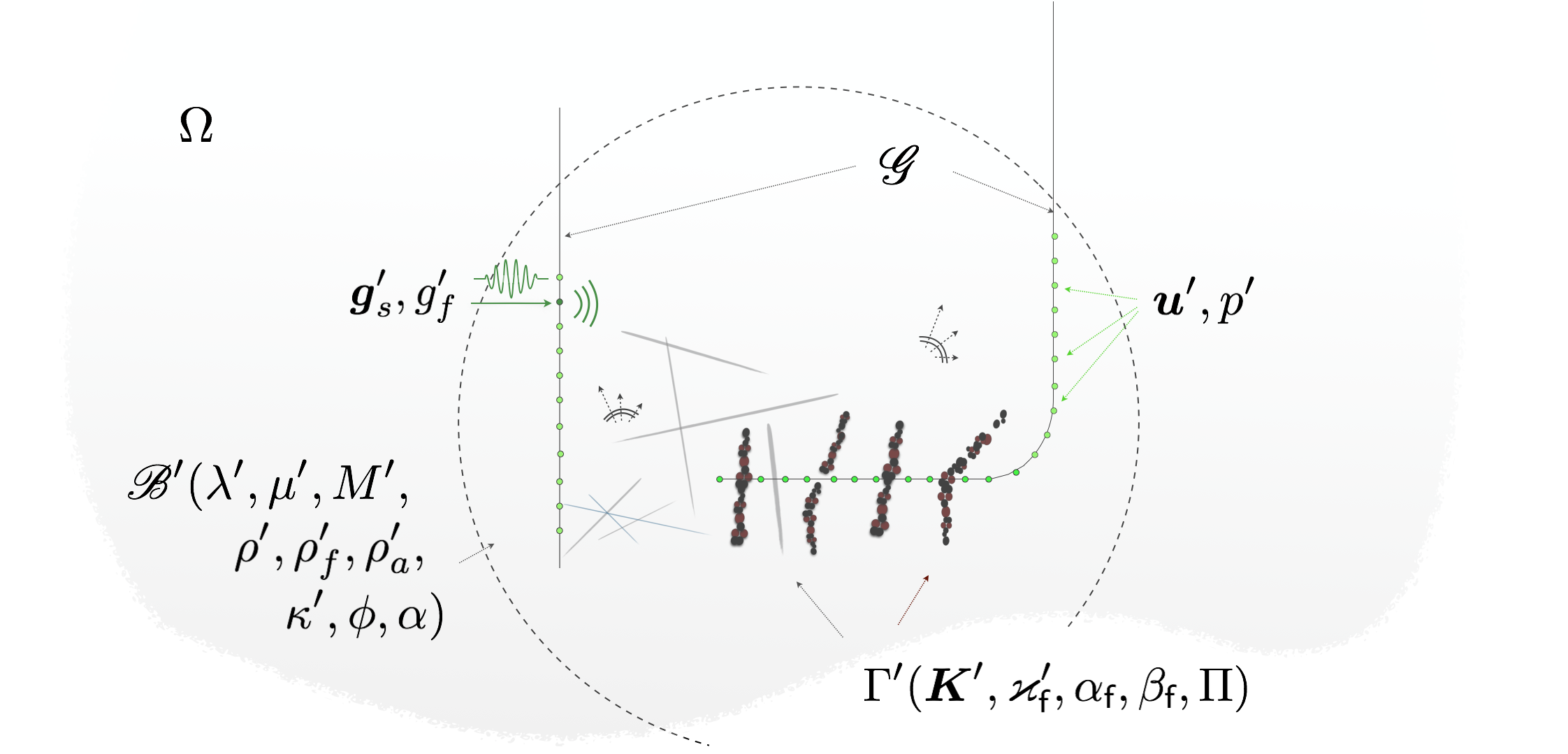} \vspace*{0mm} 
\caption{\small{Near-field illumination of a fracture network in the stimulation zone $\mathcal{B}' \subset \Omega$ featuring disjoint multiphasic discontinuities $\Gamma'\!$ characterized by heterogeneous elasticity~$\bK'(\bxi')$ and permeability~$\varkappa_{\textrm{\tiny f}}'(\bxi')$. Poroelastic incidents are induced via total body forces and fluid volumetric sources of respective densities $\bg'_s(\bxi',\omega')$ and $g'_f(\bxi',\omega')$ on a designated grid $\bxi' \in \mathcal{G}$ at frequency $\omega'$, which give rise to the scattered fields captured on $\mathcal{G}$ in the form of pore pressure $p'$ and solid displacements $\bu'\!$.}}
 \lb{fig1}
\end{figure} 

\vspace{-2mm}
 \subsection{Dimensional platform}\label{DiP} 
To enable multiphasic data inversion, the reference scales
 \vspace{-0mm}
\beq\lb{refSc}
\mu_r \,=\, \lambda', \quad \rho_r \,=\, \rho', \quad \ell_r \,=\, 2\pi \sqrt{\dfrac{\mu'}{(1-\phi){\rho}\xxs'_s {\omega'}^2}}, 
 \vspace{-0mm}
\eeq
are respectively defined for stress, mass density, and length. Here, $\rho_s'$ denotes the solid-phase density. Note that $\ell_r$ represents the drained shear wavelength. In this setting, all physical quantities are rendered \emph{dimensionless}~\cite{Scaling2003} as follows

 \vspace{-0mm}
\beq\nonumber
\begin{aligned}
& (\bxi',\omega') ~\rightarrow~ (\bxi,\omega):= \big(\dfrac{1}{\ell_r}\bxi', \sqrt{\dfrac{\rho_r}{\mu_r}}\ell_r \omega' \big), \\*[0.25mm]
\end{aligned}      
 \vspace{-0mm}
\eeq
 \vspace{-1mm}
\beq\lb{Dp}
\begin{aligned}
& \mathcal{B}'(\lambda',\mu',M',\rho',\rho_f',\rho_a',\kappa',\phi,\alpha) ~\rightarrow~   \\*[0.25mm]
& \mathcal{B}(\lambda,\mu,M,\rho,\rho_f,\rho_a,\kappa,\phi,\alpha) :=  \mathcal{B}(\dfrac{\lambda'}{\mu_r},\dfrac{\mu'}{\mu_r},\dfrac{M'}{\mu_r},\dfrac{\rho'}{\rho_r},\dfrac{\rho_f'}{\rho_r},\dfrac{\rho_a'}{\rho_r},\dfrac{\sqrt{\mu_r \rho_r}}{\ell_r}\kappa',\phi,\alpha), \\*[0.25mm]
& {\Gamma}'(\bK',\varkappa_{\textrm{\tiny f}}',\alpha_{\textrm{\tiny f}},\beta_{\textrm{\tiny f}},\Pi) ~\rightarrow~ {\Gamma}(\bK,\varkappa_{\textrm{\tiny f}},\alpha_{\textrm{\tiny f}},\beta_{\textrm{\tiny f}},\Pi) := {\Gamma}(\dfrac{\ell_r}{\mu_r}\bK', \sqrt{\mu_r \rho_r} \varkappa_{\textrm{\tiny f}}',\alpha_{\textrm{\tiny f}},\beta_{\textrm{\tiny f}},\Pi), \\*[0.0mm]
& (\bg_s',g_f') ~\rightarrow~ (\bg_s,g_f):= \big(\dfrac{\ell_r}{\mu_r} \bg_s', g_f' \big), \\*[0.5mm]
& (\bu',p') ~\rightarrow~ (\bu,p):= \big(\dfrac{\bu'}{\ell_r}, \dfrac{p'}{\mu_r} \big).
\end{aligned}      
 \vspace{-1mm}
\eeq
 
 \vspace{-2mm}
\subsection{Field equations}

The incident field $(\bu\ff,p\ff) \in H^1_{\tiny\textrm{loc}}({\R^3}\!\setminus\! \overline{\mathcal{G}})^3 \nxs\times\nxs H^1_{\tiny\textrm{loc}}({\R^3}\!\setminus\! \overline{\mathcal{G}})$ generated by $(\bg_s,g_f) \in L^2(\mathcal{G})^3 \nxs\times\nxs L^2(\mathcal{G})$ in the background is governed by
 \vspace{-1mm}
\beq\lb{uf}
\begin{aligned}
&\nabla \exs\sip\exs (\bC \exs \colon \! \nabla \bu\ff)(\bxi) \,-\, (\alpha-\dfrac{\rho_f}{\gamma}) \nabla p\ff(\bxi)  \,+\, \omega^2 (\rho-\dfrac{\rho_f^2}{\gamma}) \bu\ff(\bxi)\,=\, -\int_\mathcal{G} \delta(\by-\bxi) \bg_s(\by) \exs \textrm{d}{\by}, \,\, &  \bxi \in {\R^3}\!\setminus\! \overline{\mathcal{G}}, \\*[0.75mm] 
&\dfrac{1}{\gamma\omega^2} \nabla^2p\ff(\bxi) \,+\, {M}^{-1} p\ff(\bxi) \,+\, (\alpha-\dfrac{\rho_f}{\gamma}) \nabla \exs\sip\exs \bu\ff(\bxi)  \,=\, -   \int_\mathcal{G} \delta(\by-\bxi) g_f(\by) \exs \textrm{d}{\by},   &  \bxi \in {\R^3}\!\setminus\! \overline{\mathcal{G}}, 
\end{aligned}   
 \vspace{-0mm}
\eeq  
where $\bC = \lambda\exs\bI_2\!\otimes\bI_2 + 2\mu\exs\bI_4$ is the fourth-order drained elasticity tensor with $\bI_m \,(m\!=\!2,4)$ denoting the $m$th-order symmetric identity tensor, and 
 \vspace{-1mm}
\beq\lb{gam}
\gamma ~=~ \frac{\rho_a}{\phi^2} \,+\, \frac{\rho_f}{\phi} \,+\, \frac{\textrm{i}}{\omega \kappa}.
 \vspace{-1mm}
\eeq
 
The interaction of $(\bu\ff,p\ff)$ with $\Gamma$ gives rise to the scattered field $(\bu,p) \in H^1_{\tiny\textrm{loc}}({\R^3\!\setminus\!\overline{\Gamma}})^3 \nxs\times\nxs H^1_{\tiny\textrm{loc}}({\R^3\!\setminus\!\overline{\Gamma}})$ solving 
 \vspace{-1mm}
\beq\lb{GE}
\begin{aligned}
&\nabla \exs\sip\exs (\bC \exs \colon \! \nabla \bu) \,-\, (\alpha-\dfrac{\rho_f}{\gamma}) \nabla p  \,+\, \omega^2 (\rho-\dfrac{\rho_f^2}{\gamma}) \bu\,=\, \bzero \quad& \text{in}&\,\,\, {\R^3\!\setminus\!\overline{\Gamma}}, \\*[0.75mm]
&\dfrac{1}{\gamma\omega^2} \nabla^2p \,+\, {M}^{-1} p \,+\, (\alpha-\dfrac{\rho_f}{\gamma}) \nabla \exs\sip\exs \bu  \,=\, 0  \quad & \text{in}&\,\,\,  {\R^3\!\setminus\!\overline{\Gamma}}, \\*[1.25mm]
& \bt \,+\, \tilde{\alpha}_{\textrm{\tiny f}} \llangle p \rrangle \bn \,=\, \bK \dbr{\bu} \,-\, \bt\ff \,-\, \tilde{\alpha}_{\textrm{\tiny f}} \exs p\ff  \bn \quad& \text{on}&\,\,\,  {\Gamma}, \\*[1.25mm]
&   \llangle p \rrangle \,+\, {\beta}_{\textrm{\tiny f}} \exs \bt\sip\bn \,=\, - \dfrac{\textrm{k}_n {\beta}_{\textrm{\tiny f}}}{\Pi \alpha_{\textrm{\tiny f}}} \dbr{q} \,-\,  p\ff \,-\, {\beta}_{\textrm{\tiny f}} \exs \bt\ff\sip\bn  \quad& \text{on}&\,\,\,  {\Gamma}, \\*[1mm]
&  \llangle q \rrangle \,=\, \dfrac{\varkappa_{\textrm{\tiny f}}}{\textrm{i}\omega\Pi} \dbr{p} \,-\, q\ff, \quad \dbr{\bt}=\bzero \quad & \text{on}&\,\,\,  {\Gamma}, 
\end{aligned}   
 \vspace{1mm}
\eeq 
where $(\bt,q) \in H^{-1/2}(\Gamma)^3 \nxs\times\nxs H^{-1/2}(\Gamma)$ and $(\bt\ff,q\ff) \in H^{-1/2}(\Gamma)^3 \nxs\times\nxs H^{-1/2}(\Gamma)$ are defined by 
\beq\lb{Neu}
\begin{aligned}
&[\bt,q](\bu,p) \,:=\, \big[\bn \nxs\cdot\nxs \bC \colon \!\nxs \nabla \bu - \alpha p \exs \bn, \, \dfrac{1}{\gamma\omega^2}(\nabla p \sip \bn - \rho_f \omega^2 \bu\sip\bn) \big]  \,\,\, & \text{on} \,\,\,  {\Gamma}, \\*[1mm]
&[\bt\ff,q\ff](\bu\ff,p\ff) \,:=\, \big[\bn \nxs\cdot\nxs \bC \colon \!\nxs \nabla \bu\ff - \alpha p\ff  \bn, \, \dfrac{1}{\gamma\omega^2}(\nabla p\ff\nxs \sip \bn - \rho_f \omega^2 \bu\ff\nxs\sip\bn) \big] \,\,\, & \text{on} \,\,\,  {\Gamma}; 
\end{aligned}
\vspace*{-1mm}
\eeq
the unit normal $\bn = \bn^-$ on~$\Gamma$ is explicitly identified in Section~\ref{Fsp}; 
\[
\tilde{\alpha}_{\textrm{\tiny f}} \,:=\, \dfrac{\alpha_{\textrm{\tiny f}} \Pi}{1-\alpha_{\textrm{\tiny f}} \beta_{\textrm{\tiny f}}(1-\Pi)}; 
\]
$\big( \dbr{\bu},\dbr{p},\dbr{q} \big) = \big( \bu^+\!-\bu^-, p^+\! - p^-, q^+\! - q^- \big)$ signifies the jump in~$( \bu,p,q )$ across~$\Gamma$, while $$\big(  \llangle{\bu}\rrangle,\llangle{p}\rrangle,\llangle{q}\rrangle \big)  = \dfrac{1}{2} \big(  \bu^+\!+\bu^-, p^+\! + p^-, q^+\! + q^- \big) $$ is the respective mean fields on~$\Gamma\,$; $\bK=\bK(\bxi)$, according to~\cite{naka2007}, is a \emph{symmetric} and possibly \emph{complex-valued} matrix of specific stiffnesses which in the fracture's local coordinates $(\be_1,\be_2,\bn)(\bxi)$ may be described as the following
\beq\lb{contact}
\bK \,:=\, \textrm{k}_t \exs \be_1 \!\otimes\nxs \be_1 \,+\, \textrm{k}_t \exs \be_2 \!\otimes\nxs \be_2 \,+\, \dfrac{\tilde{\alpha}_{\textrm{\tiny f}}\textrm{k}_n}{{\alpha}_{\textrm{\tiny f}}\Pi} \exs \bn \!\otimes\nxs \bn \quad \text{on} \,\,\,  {\Gamma},
\eeq
wherein $\textrm{k}_t(\bxi)$ and $\textrm{k}_n(\bxi)$ are known respectively as the tangential and normal drained elastic stiffnesses. It should be mentioned that the contact condition across $\Gamma$ makes use of the generalized Schoenberg's model for a poroelatic interface of finite permeability~\cite{naka2007}. 
 
\begin{rem}
The contact law at high-permeability interfaces~\cite{naka2007} is reduced to the following form
 \beq\lb{HPI}\nonumber
 \begin{aligned}
& \bt \,+\, {\alpha}_{\textrm{\tiny \emph{f}}} \exs  p  \bn \,=\, \bK_{\nxs \infty} \dbr{u} \,-\, \bt\ff \,-\, {\alpha}_{\textrm{\tiny \emph{f}}} \exs p\ff  \bn \quad& \text{on}&\,\,\,  {\Gamma}, \\*[1mm]
&   p  \,+\, {\beta}_{\textrm{\tiny \emph{f}}} \exs \bt\sip\bn \,=\, - \textrm{\emph{k}}_n \dfrac{ {\beta}_{\textrm{\tiny \emph{f}}}}{\alpha_{\textrm{\tiny \emph{f}}}} \dbr{q} \,-\,  p\ff \,-\, {\beta}_{\textrm{\tiny \emph{f}}} \exs \bt\ff\sip\bn  \quad& \text{on}&\,\,\,  {\Gamma}, \\*[0.75mm]
&  \dbr{p} \,=\, 0, \,\,\,\, \dbr{\bt} = \bzero \quad & \text{on}&\,\,\,  {\Gamma}. 
\end{aligned}  
 \eeq
 where $\bK_{\nxs \infty}  := \textrm{\emph{k}}_t \exs \be_1 \!\otimes\nxs \be_1 \exs+\exs \textrm{\emph{k}}_t \exs \be_2 \!\otimes\nxs \be_2 \exs+\exs \textrm{\emph{k}}_n \exs \bn \!\otimes\nxs \bn$ on ${\Gamma}$.
\end{rem}

The formulation of the direct scattering problems may now be completed by requiring that~$(\bu,p)$ and $(\bu\ff,p\ff)$ satisfy the radiation condition as $|\bxi| \rightarrow \infty$~\cite{norr1985}. On uniquely decomposing the incident and scattered fields into two irrotational parts~(${\textrm{p}_1}$, ${\textrm{p}_2}$) and a solenoidal part~(${\textrm{s}}$) as 
\beq\lb{dec}
 (\bu\ff,p\ff)  = (\bu\ff_{\textrm{p}_1\!},p\ff_{\textrm{p}_1\nxs}) + (\bu\ff_{\textrm{p}_2\!},p\ff_{\textrm{p}_2\nxs}) + (\bu\ff_{\textrm{s}},p\ff_{\textrm{s}}), \quad
 (\bu,p)   = (\bu_{\textrm{p}_1\!},p_{\textrm{p}_1\nxs}) + (\bu_{\textrm{p}_2\!},p_{\textrm{p}_2\nxs}) + (\bu_{\textrm{s}},p_{\textrm{s}}),
\eeq  
 the radiation condition can be stated as 
\beq\lb{KS}
\begin{aligned}
& \frac{\partial{\textrm{\bf u}}_\varsigma}{\partial r} - \text{i} k_\varsigma {\textrm{\bf u}}_\varsigma = o\big(r^{-1} e^{-a_\varsigma r}\big)   \quad \text{as} ~~r:=|\bxi|\to \infty, \quad a_\varsigma:=\mathfrak{I}(k_\varsigma), \,\,\, \varsigma = {\textrm{p}_1},{\textrm{p}_2},{\textrm{s}}, \,\,\, \textrm{\bf u} = \bu\ff, \bu, \\*[0.5mm]
& \frac{\partial{\textrm{p}}_\varsigma}{\partial r} - \text{i} k_\varsigma {\textrm{p}}_\varsigma = o\big(r^{-1} e^{-a_\varsigma r}\big)   \quad \text{as} ~~r\to \infty, \qquad  \varsigma = {\textrm{p}_1},{\textrm{p}_2}, \,\,\, \textrm{p} = p\ff, p,
\end{aligned}
\eeq  
uniformly with respect to $\hat\bxi:=\bxi/r$ where $k_\varsigma$, $\varsigma = {\textrm{p}_1},{\textrm{p}_2},{\textrm{s}}$, is the complex-valued wavenumber associated with the slow and fast p-waves and the transverse s-wave~\cite{bour1992}. 

\vspace{-3mm}
\subsection{Function spaces}\label{Fsp}
For clarity of discussion, it should be mentioned that the support of $\Gamma$ may be decomposed into $N$ smooth open subsets $\Gamma_n \!\subset \Gamma$, $n=1,\ldots N$, such that $\Gamma \!=\! {\textstyle \bigcup_{n =1}^{N}} \Gamma_n$. The support of $\Gamma_n$ may be arbitrarily extended to a closed Lipschitz surface $\partial \text{\sf D}_n$ of a bounded simply connected domain $\text{\sf D}_n$. In this setting, the unit normal vector $\bn$ to $\Gamma_n$ coincides with the outward normal vector to $\partial \text{\sf D}_n$. Let $\text{\sf D} = {\textstyle \bigcup_{n =1}^{N}} \text{\sf D}_n$ be a multiply connected Lipschitz domain of bounded support such that $\Gamma \subset \partial \text{\sf D}$, then $\Gamma$ is assumed to be an open set relative to $\partial \text{\sf D}$ with a positive surface measure, and the closure of $\Gamma$ is denoted by $\overline{\Gamma} \colon \!\!\! =  \Gamma \cup \partial\Gamma$.   
Following~\cite{McLean2000}, we define
\beq\lb{funS2}
\begin{aligned}
&H^{\pm \frac{1}{2}}(\Gamma) ~:=~\big\lbrace f\big|_{\Gamma} \! \colon \,\,\, f \in H^{\pm \frac{1}{2}}(\partial \text{\sf D}) \big\rbrace, \\*[0.0 mm]
& \tilde{H}^{\pm \frac{1}{2}}(\Gamma) ~:=~\big\lbrace  f \in H^{\pm\frac{1}{2}}(\partial \text{\sf D}) \colon  \,\,\, \text{supp}(f) \subset \overline{\Gamma} \exs \big\rbrace,
\end{aligned}
\eeq
and recall that $H^{-1/2}(\Gamma)$ and $\tilde{H}^{-1/2}(\Gamma)$ are respectively the dual spaces of $\tilde{H}^{1/2}(\Gamma)$ and $H^{1/2}(\Gamma)$. Accordingly, the following embeddings hold
\beq\lb{embb}
\tilde{H}^{\frac{1}{2}}(\Gamma) \,\subset\, H^{\frac{1}{2}}(\Gamma) \,\subset\, L^2(\Gamma) \,\subset\, \tilde{H}^{-\frac{1}{2}}(\Gamma) \,\subset\, H^{-\frac{1}{2}}(\Gamma).
\eeq

Note that since $(\bu,p) \in H^1_{\tiny\textrm{loc}}({\R^3}\nxs\setminus\nxs \overline{\Gamma})^3 \nxs\times\nxs H^1_{\tiny\textrm{loc}}({\R^3}\nxs\setminus\nxs\overline{\Gamma})$, then by trace theorems $(\dbr{\bu},\dbr{p})\in\tilde{H}^{1/2}(\Gamma)^3 \nxs\times\nxs \tilde{H}^{1/2}(\Gamma)$, $(\bt,q) \in H^{-1/2}(\Gamma)^3 \nxs\times\nxs H^{-1/2}(\Gamma)$, and $\dbr{q} \in \tilde{H}^{-1/2}(\Gamma)$.

\begin{rem}[Wellposedness of the forward scattering problem]\label{R_Wellp}
In light of~\eqref{GE}, let us define the interface operator $\mathfrak{P}: \tilde{H}^{{1}/{2}}(\Gamma)^3\times \tilde{H}^{{1}/{2}}(\Gamma) \times \tilde{H}^{-{1}/{2}}(\Gamma) \rightarrow {H}^{-{1}/{2}}(\Gamma)^3\times {H}^{-{1}/{2}}(\Gamma) \times {H}^{{1}/{2}}(\Gamma)$ such that 
\beq\lb{frakP}
\mathfrak{P}(\dbr{\bu},\dbr{\xxs p \xxs},-\dbr{q \xxs})\,:=\,(\bt+\bt\ff,\llangle{q}\rrangle+q\ff,\llangle{\xxs p}\rrangle+p\ff) \qquad \text{on}\quad  {\Gamma}.
\eeq 
Then,~\eqref{GE}-\eqref{KS} is wellposed provided that $\mathfrak{P}$ is bounded and 
 \beq\lb{ImP}
 \Im \dualGA{\mathfrak{P}\bphi}{\bphi} \,\leqslant\, 0,\, \qquad \forall \exs \bphi \in \tilde{H}^{{1}/{2}}(\Gamma)^3\times \tilde{H}^{{1}/{2}}(\Gamma) \times \tilde{H}^{-{1}/{2}}(\Gamma):~ \bphi \neq \bzero,
 \eeq  
 given the duality product \[
\langle \cdot, \cdot \rangle_{\Gamma} ~=~ \big\langle {H}^{-\frac{1}{2}}(\Gamma)^3 \nxs\times {H}^{-\frac{1}{2}}(\Gamma) \times {H}^{\frac{1}{2}}(\Gamma), \, \tilde{H}^{\frac{1}{2}}(\Gamma)^3 \nxs\times \tilde{H}^{\frac{1}{2}}(\Gamma) \times \tilde{H}^{-\frac{1}{2}}(\Gamma) \big\rangle.
\]
Detailed analysis is included in~\ref{Wellp}.
\end{rem}

\subsection{Poroelastic scattering}\lb{Prelim}   
\renewcommand{\OOd}{{\mathcal{G}}}
\renewcommand{\pff}{\btu}

In this section, the poroelastic scattering operator is introduced and its first and second factorizations are formulated. In what follows, the Einstein's summation convention applies over the repeated indexes. 

\vspace{-2mm}
\subsubsection{Incident wave function \!\!} 
Observe that for a given density $L^2(\mathcal{G})^3\times L^2(\mathcal{G}) \ni \bg := (\bg_s,g_f)$, the poroelastic incident field $(\bu\ff,p\ff) \in H^1_{\tiny\textrm{loc}}({\R^3}\nxs\setminus\nxs \overline{\mathcal{G}})^3 \times H^1_{\tiny\textrm{loc}}({\R^3}\nxs\setminus\nxs\overline{\mathcal{G}})$ solving~\eqref{uf} may be recast in the integral form 
\beq\lb{HW}
\begin{aligned}
&  \left[ \!\! \begin{array}{l}
\\*[-5.45mm]
u\ff_i \\*[-0.1mm]
\exs\exs p\ff 
\end{array} \!\!\nxs \right]\!(\bxi) ~=~  \int_{\OOd} \left[ \!\! \begin{array}{ll}
\Usfr_{\!\!\! ij} \!\! &\!\! \uffr_{\nxs\! i} \\*[0.1mm]
\exs \psfr_{\!\!\! j} \!&\!\! \pffr \\
\end{array} \!\!\nxs \right]\!(\by,\bxi) \,\sip
\left[ \!\!\! \begin{array}{l}
{\exs g_s^{\xxs j}} \\*[-0.5mm]
{\exs g_f} 
\end{array} \!\!\! \right]\!\!(\by) \exs \text{d}{\by},  \qquad \bxi \in \R^3\nxs\setminus\nxs\overline{\mathcal{G}},  \qquad i,j = 1,2,3,
\end{aligned}
\eeq
whose kernel is the poroelastodynamics fundamental solution tensor provided in~\ref{fund}.~Here, $\Usfr_{\!\!\! ij}(\by,\bxi)$ $(i,j = 1,2,3)$ signifies the $i^\text{\xxs th}$ component of the solid displacement at $\bxi$ due to a unit total body force applied at $\by$ along the coordinate direction $j$, while $\psfr_{\!\!\! j}$ is the associated pore pressure. Likewise, $\uffr_{\nxs\! i}(\by,\bxi)$ $(i = 1,2,3)$ stands for the solid displacement at $\bxi$ in the $i^\text{\xxs th}$ direction due to a unit fluid source i.e.,~volumetric injection at $\by$, and $\pffr$ is the induced pore pressure.   

\vspace{-2mm}
\subsubsection{Poroelastic scattered field \!\!}
By way of the reciprocal theorem of poroelastodynamics~\cite{chen1991,scha2012}, one may show that the following integral representation holds for the scattered field $(\bu,p) \in H^1_{\tiny\textrm{loc}}({\R^3}\nxs\setminus\nxs\overline{\Gamma})^3 \nxs\times\nxs H^1_{\tiny\textrm{loc}}({\R^3}\nxs\setminus\nxs\overline{\Gamma})$ satisfying~\eqref{GE}-\eqref{KS},
\vspace{-0mm}
\beq\lb{vinf2}
\begin{aligned}
&  \left[ \!\! \begin{array}{l}
\\*[-5.85mm]
u_i \\*[-1mm]
\exs\exs p 
\end{array} \!\!\nxs \right]\!(\bxi) ~=~\! \int_{\Gamma}
\left[ \!\! \begin{array}{lll}
\\*[-5.85mm]
\Tsfr_{\!\!\! ij} \!\! &\!\! \qsfr_{\!\!\! i} \!\! &\! \psfr_{\!\!\! i} \\*[0mm]
 \tffr_{\nxs\! j} \!\! &\!\! \qffr \!\! &\!  \pffr
\end{array} \!\!\nxs \right]\!(\by,\bxi) \,\sip
\left[ \!\!\! \begin{array}{l}
\dbr{u_j} \\*[-0.5mm]
\exs\exs \dbr{\exs p\hspace{0.2mm}} \\*[-0.5mm]
 \textrm{\small -} \dbr{\hspace{0.2mm}q\hspace{0.15mm}}
\end{array} \!\!\! \right]\!\!(\by) \exs \text{d}S_{\by}, \qquad \bxi\in\R^3\nxs\setminus\nxs\overline{\Gamma}, \qquad i,j = 1,2,3,
\end{aligned} 
\eeq
whose kernel is derived from the fundamental solution tensor and provided in~\ref{fund}. More specifically, $\bTsfr$ and $\btffr$ indicate the fundamental tractions on $\Gamma$; $\bqsfr$ and $\qffr$ signify the associated relative fluid-solid displacements across $\Gamma$; $\bpsfr$ and $\pffr$ are the fundamental pore pressure solutions. 

\vspace{-2mm}
\subsubsection{Poroelastic scattering operator\!\!}

The linearity of~\eqref{GE} implies that the scattered wave function may be expressed as a linear integral operator. To observe this, let $\R^4 \ni \bd := (\bd_s,d_p)$ be a constant vector characterizing an incident point source wherein $\bd_s$ indicates the amplitude and direction of the total body force, while $d_p$ represents the magnitude of fluid volumetric source both applied at $\by \in \mathcal{G}$. The resulting scattered field is observed in terms of seismic displacements $\bu_{\bd}$ and pore pressure $p_{\bd}$ at $\bxi \in \mathcal{G}$. Now, let us define the scattering kernel $\boldsymbol{\mathfrak{S}}(\by,\bxi)\in\mathbb{C}^{4\times 4}$ so that 
\beq\lb{w-inf}
\boldsymbol{\mathfrak{S}}(\by,\bxi) \sip \bd ~:=~ (\bu_{\bd},\exs p_{\bd})(\bxi), \qquad \by,\bxi \in \mathcal{G}.
\eeq
Then, given $\bg \in L^2(\mathcal{G})^3 \times L^2(\mathcal{G})$, one may easily verify that  
\beq\lb{ffo2} 
(\bu,p)(\bxi) ~=\,  \int_{\OOd} \boldsymbol{\mathfrak{S}}(\by,\bxi) \sip \bg(\by) \,\, \text{d}S_{\by}, \qquad \by,\bxi \in \mathcal{G},
\eeq 
where $(\bu,p) \in L^2(\mathcal{G})^3 \times L^2(\mathcal{G})$ satisfies~\eqref{GE}-\eqref{KS}. 

\noindent Accordingly, one may define the poroelastic scattering operator $\Lambda: L^2(\OOd)^3 \times L^2(\mathcal{G}) \to L^2(\OOd)^3 \times L^2(\mathcal{G})$ by
\beq\lb{ffo0}
\Lambda(\bg) ~=~ (\bu,p)|_{\mathcal{G}}.
\eeq 

\vspace{-3mm}
\subsubsection{Factorization of $\Lambda \!\!\!$}
With reference to~\eqref{uf} and~\eqref{Neu}, let us define the operator $\mathcal{S} \colon L^2(\OOd)^3 \times L^2(\mathcal{G}) \rightarrow H^{-1/2}(\Gamma)^3 \times H^{-1/2}(\Gamma) \times H^{1/2}(\Gamma)$ given by 
\beq\lb{oH}
\mathcal{S}(\bg) ~:=~ (\bt^\iota, q^\iota, p^\iota) \quad~~ \text{on}\quad \Gamma,
\eeq  
where $\bt^\iota$ specifies the incident field traction, $q^\iota$ is the specific relative fluid-solid displacement across $\Gamma$, and $p^\iota$ is the incident pore pressure on $\Gamma$. Next, define $\mathcal{P} \colon H^{-1/2}(\Gamma)^3 \nxs\times\nxs H^{-1/2}(\Gamma) \nxs\times\nxs H^{1/2}(\Gamma) \rightarrow L^2(\OOd)^3\times L^2(\mathcal{G})$ as the map taking the incident traction, relative flow and pressure $(\bt^\iota, q^\iota, p^\iota)$ over $\Gamma$ to the multiphase scattered data $(\bu,p) \in L^2(\mathcal{G})^3 \times L^2(\mathcal{G})$ satisfying \eqref{GE}-\eqref{KS}. Then, the scattering operator~\eqref{ffo0} may be factorized as  
\beq\lb{fac1}
\Lambda ~=~ \mathcal{P} \mathcal{S}.
\eeq   
\begin{lemma}\label{H*}
The adjoint operator $\mathcal{S}^* \colon \tilde{H}^{1/2}(\Gamma)^3 \nxs\times\nxs \tilde{H}^{1/2}(\Gamma) \nxs\times\nxs \tilde{H}^{-1/2}(\Gamma) \rightarrow L^2(\OOd)^3 \nxs\times\nxs L^2(\mathcal{G})$ takes the form
\beq\lb{Hstar}
\mathcal{S}^*(\ba) ~:=~ (\bu_{\ba}, p_{\ba})(\bxi), \qquad \bxi \in \mathcal{G}, 
\eeq  
where $(\bu_{\ba}, p_{\ba}) \in H^1_{\tiny\textrm{loc}}({\R^3}\nxs\setminus\nxs\overline{\Gamma})^3 \nxs\times\nxs H^1_{\tiny\textrm{loc}}({\R^3}\nxs\setminus\nxs\overline{\Gamma})$ solves
\vspace{-1mm}
\beq\lb{up_a}
\begin{aligned}
&\nabla \exs\sip\exs (\bC \exs \colon \! \nabla \bu_{\ba}) \,-\, (\alpha-\dfrac{\rho_f}{\bar{\gamma}}) \nabla p_{\ba}  \,+\, \omega^2 (\rho-\dfrac{\rho_f^2}{\bar{\gamma}}) \bu_{\ba}\,=\, \bzero \quad& \text{in}&\,\,\, {\R^3\nxs\setminus\nxs\overline{\Gamma}}, \\*[1mm]
&\dfrac{1}{\bar{\gamma}\omega^2}\nabla^2p_{\ba} \,+\, {M}^{-1} p_{\ba} \,+\, (\alpha-\dfrac{\rho_f}{\bar{\gamma}}) \nabla \exs\sip\exs \bu_{\ba}  \,=\, 0  \quad & \text{in}&\,\,\,  {\R^3}\nxs\setminus\nxs\overline{\Gamma}, \\*[2mm]
& (\dbr{\bu_{\ba}},\dbr{p_{\ba}},-\dbr{q_{\ba}}) \,=\, \ba, \quad \dbr{\bt_{\ba}} = \bzero \quad& \text{on}&\,\,\,  {\Gamma}.
\end{aligned}   
\eeq 

Here, the `bar' indicates complex conjugate, and
\[
[\bt_{\ba},q_{\ba}](\bu_{\ba},p_{\ba}) \,:=\, \big[\bn \nxs\cdot\nxs \bC \colon \!\nxs \nabla \bu_{\ba} - \alpha \exs p_{\ba}  \bn, \, \dfrac{1}{\bar{\gamma}\omega^2}(\nabla p_{\ba} \sip\exs \bn - \rho_f \omega^2 \bu_{\ba} \sip\exs \bn) \big] \quad  \text{on} \,\,\,  {\Gamma}. 
\]
Note that $(\bu_{\ba},p_{\ba})$ satisfy the radiation condition as $|\bxi|\to\infty$, similar to the complex conjugate of~\eqref{KS}. 
\end{lemma}
\begin{proof}
\textcolor{black}{see~\ref{H*pruf}}.
\end{proof}
On the basis of~\eqref{vinf2} and~(\ref{Hstar}), the operator $\mathcal{P}$ can be further decomposed as $\mathcal{P}=\bar{\mathcal{S}}^* T$ where the middle operator $T\colon H^{-1/2}(\Gamma)^3 \times H^{-1/2}(\Gamma) \times H^{1/2}(\Gamma) \rightarrow \tilde{H}^{1/2}(\Gamma)^3 \times \tilde{H}^{1/2}(\Gamma) \times \tilde{H}^{-1/2}(\Gamma)^3$ is given by
\beq\lb{T}
T(\bt\ff,q\ff,p\ff) ~:=~ \big( \dbr{\bu}, \dbr{\exs p \exs}, -\dbr{\exs q\exs} \big)   \qquad \text{on}\,\,\,  {\Gamma},
\eeq
such that~$(\bu,p) \in H^1_{\tiny\textrm{loc}}({\R^3}\nxs\setminus\nxs\overline{\Gamma})^3 \nxs\times\nxs H^1_{\tiny\textrm{loc}}({\R^3}\nxs\setminus\nxs\overline{\Gamma})$ satisfies~\eqref{vinf2} and $q$ is computed from~\eqref{Neu}. Thanks to this new decomposition of $\mathcal{P}$, a second factorization of $\Lambda \colon L^2(\OOd)^3 \times L^2(\OOd) \rightarrow L^2(\OOd)^3 \times L^2(\OOd)$ is obtained as 
\beq\lb{fact} 
\Lambda ~=~ \bar{\mathcal{S}}^* \exs T \exs \mathcal{S}. 
\eeq

It is worth noting that the asymmetry of the second factorization~\eqref{fact} stems from the non-selfadjoint nature of the Biot system which is evident from~\eqref{up_a}.    

\section{Key properties of the poroelastodynamic operators}\label{SFS}

\begin{lemma}\lb{H*p}
Operator $\mathcal{S}^* \colon \tilde{H}^{1/2}(\Gamma)^3 \nxs\times\nxs \tilde{H}^{1/2}(\Gamma) \nxs\times\nxs \tilde{H}^{-1/2}(\Gamma) \rightarrow L^2(\OOd)^3 \nxs\times\nxs L^2(\mathcal{G})$ in Lemma~\ref{H*} is compact, injective, and has a dense range.  
\end{lemma}
\begin{proof} 
See~\ref{S*pruf}.
\end{proof}

\begin{lemma}\lb{I{T}>0}
Operator $T\colon H^{-1/2}(\Gamma)^3 \times H^{-1/2}(\Gamma) \times H^{1/2}(\Gamma) \rightarrow \tilde{H}^{1/2}(\Gamma)^3 \times \tilde{H}^{1/2}(\Gamma) \times \tilde{H}^{-1/2}(\Gamma)^3$ in~(\ref{T}) is bounded and satisfies
 \beq\lb{pos-IT}
 \Im \dualGA{\bpsi}{T\bpsi} <0,\, \quad \forall \exs \bpsi \in H^{-1/2}(\Gamma)^3 \times H^{-1/2}(\Gamma) \times H^{1/2}(\Gamma):~ \bpsi \neq \bzero.
 \eeq
\end{lemma}
\begin{proof}
See~\ref{I{T}>0p}.
\end{proof}

\begin{lemma}\lb{T-invs0}
Operator $T\colon H^{-1/2}(\Gamma)^3 \times H^{-1/2}(\Gamma) \times H^{1/2}(\Gamma) \rightarrow \tilde{H}^{1/2}(\Gamma)^3 \times \tilde{H}^{1/2}(\Gamma) \times \tilde{H}^{-1/2}(\Gamma)^3$:~(i)~has a bounded (and thus continuous) inverse, and~(ii)~is coercive, i.e., there exists constant $c\!>\!0$ independent of~$\bpsi$ such that
\beq\lb{co-T0}
|\langle \bpsi, \, T (\bpsi) \rangle| \,\,\geqslant\,\, \textrm{c} \nxs \norms{\bpsi}_{H^{-1/2}(\Gamma)^3 \times H^{-1/2}(\Gamma) \times H^{1/2}(\Gamma)}^2, \quad  \forall\bpsi\in H^{-1/2}(\Gamma)^3 \times H^{-1/2}(\Gamma) \times H^{1/2}(\Gamma).
\eeq 
\end{lemma} 
\begin{proof}
See~\ref{T-invs0p}.
\end{proof}

\begin{lemma}\lb{comp_G}
Operator $\mathcal{P} = \bar{\mathcal{S}}^* T \,\colon\nxs H^{-1/2}(\Gamma)^3 \nxs\times\nxs H^{-1/2}(\Gamma) \nxs\times\nxs H^{1/2}(\Gamma) \rightarrow L^2(\OOd)^3\times L^2(\mathcal{G})$ is compact over $H^{-1/2}(\Gamma)^3 \nxs\times\nxs H^{-1/2}(\Gamma) \nxs\times\nxs H^{1/2}(\Gamma)$. 
\end{lemma}
\begin{proof}
The claim follows immediately from Lemmas~\ref{H*p} and~\ref{I{T}>0} establishing, respectively, the compactness of $\mathcal{S}^*$ and the boundedness of $T$.
 \end{proof}
 
 \begin{lemma}\lb{FF_op}
The poroelastic scattering operator $\Lambda: L^2(\OOd)^3 \times L^2(\mathcal{G}) \to L^2(\OOd)^3 \times L^2(\mathcal{G})$ is injective, compact and has a dense range.  
\end{lemma}
\begin{proof}
See~\ref{FF_opp}.
\end{proof}

\vspace{-2mm}
\section{Inverse poroelastic scattering} \label{SSA}

This section presents an adaptation of the key results from sampling approaches to inverse scattering for the problem of poroelastic-wave imaging of finitely permeable interfaces. The fundamental idea stems from the nature of solution $\bg = (\bg_s,g_f) \in L^2(\mathcal{G})^3\times L^2(\mathcal{G})$ to the poroelastic scattering equation 
\beq\lb{FF}
\Lambda \bg ~=~ \bPhi_{\text{\tiny L}}, \qquad \Lambda ~=~ \mathcal{P} \mathcal{S} ~=~ \bar{\mathcal{S}}^* \exs T \exs \mathcal{S}, 
\eeq
where $\bPhi_{\text{\tiny L}}$ is the near-field pattern of a trial poroelastodynamic field specified by Definition~\ref{phi-infinity}.   

\begin{defn} \lb{phi-infinity}
With reference to~\eqref{vinf2}, for every admissible density $\ba\!\in\!\tilde{H}^{1/2}(L)^3 \nxs\times\nxs \tilde{H}^{1/2}(L)\nxs\times\nxs \tilde{H}^{-1/2}(L)$ specified over a smooth, non-intersecting trial interface $L\!\subset\! \mathcal{B}\!\subset\!\R^3$, the induced near-field pattern $\bPhi_{\text{\tiny L}} \colon \tilde{H}^{1/2}(L)^3 \nxs\times\nxs \tilde{H}^{1/2}(L)\nxs\times\nxs \tilde{H}^{-1/2}(L) \rightarrow L^2(\mathcal{G})^3\times L^2(\mathcal{G})$ is given by
\beq\lb{Phi-inf}
\begin{aligned}
\bPhi_{\text{\tiny \emph{L}}}(\ba)(\bxi) \,:=\,  (\bu_{\xxs\text{\tiny \emph{L}}},p_{\xxs\text{\tiny \emph{L}}})(\bxi) \,= \int_L \, \bSig(\by,\bxi) \sip \ba(\by) \,\, \text{d}S_{\by}, \quad \bSig \,=\, \text{\emph{$\left[ \!\! \begin{array}{lll}
\\*[-5.85mm]
\bTsfr \!\! &\!\! \bqsfr \!\! &\! \bpsfr \\*[0mm]
 \btffr \!\! &\!\! \qffr \!\! &\!  \pffr
\end{array} \!\!\nxs \right]$}}, \quad \bxi \in \mathcal{G}.  
\end{aligned}
\eeq
\end{defn}
\begin{rem}\lb{LSMrem}
\textcolor{black}{In light of~\eqref{FF}, one may interpret the philosophy of sampling-based waveform tomography as the following. Let ${\sf L}\!\subset\!\mathcal{B}$ (containing the origin) denote a  poroelastic discontinuity whose characteristic size is small relative to the length scales describing the forward scattering problem, and let $L=\bz\!+\bR{\sf L}$ where $\bz\!\in\mathcal{B}$ and $\bR\!\in\!U(3)$ is a unitary rotation matrix. Given a density $\ba\!\in\!\tilde{H}^{1/2}(L)^3 \nxs\times\nxs \tilde{H}^{1/2}(L)\nxs\times\nxs \tilde{H}^{-1/2}(L)$, solving the scattering equation~\eqref{FF} over a grid of trial pairs $(\bz,\bR)$ sampling $\mathcal{B}\times U(3)$ is simply an effort to probe the range of operator $\Lambda$ (or that of $\bar{\mathcal{S}}^*$), through synthetic reshaping of the  illuminating wavefront, for fingerprints in terms of~$\bPhi_{\text{\tiny \emph{L}}}$. As shown by Theorems~\ref{TR2}, such fingerprint is found in the data if and only if~$L\subset\Gamma$. Otherwise, the norm of any approximate solution to~(\ref{FF}) can be made arbitrarily large, which provides a criterion for reconstructing~$\Gamma$.}          
\end{rem}

\begin{theorem}\lb{TR1} 
Under the assumptions of~Remark~\ref{R_Wellp} for the wellposedness of the forward scattering problem, for \emph{every} smooth and non-intersecting trial dislocation $L\!\subset\! \mathcal{B} \!\subset\! \R^3$ and density profile~$\,\ba(\bxi)\!\in\!\tilde{H}^{1/2}(L)^3 \nxs\times\nxs \tilde{H}^{1/2}(L)\nxs\times\nxs \tilde{H}^{-1/2}(L)$, one has
\[
\bPhi_{\text{\tiny L}} \in Range(\bar{\mathcal{S}}^*) ~~ \iff ~~ L \subset \Gamma.
\]   
\end{theorem}
\begin{proof}
The argument directly follows that of~\cite[Theorem 6.1]{pour2017}.
\end{proof}

On the basis of Theorem~\ref{TR1}, one arrives at the following statement which inspires the sampling-based poroelastic imaging indicators.    
\begin{theorem}\lb{TR2}
Under the assumptions of~Theorem~\ref{TR1},
\begin{itemize}
\item~If $L \subset \Gamma$, there exists a density vector $\bg_\epsilon^L\!\in L^2(\OOd)^3\nxs\times\nxs L^2(\OOd)$ such that $\|\Lambda\exs\bg_\epsilon^L-\bPhi_{\text{\tiny L}}\|_{L^2(\OOd)^3\nxs\times\nxs L^2(\OOd)} \leqslant\epsilon$ and
\\
$\limsup\limits_{\epsilon \rightarrow 0} \|\mathcal{S}\bg_\epsilon^L\|_{H^{-1/2}(\Gamma)^3 \times H^{-1/2}(\Gamma) \times H^{1/2}(\Gamma)}<\infty$.

\item~If $L \not\subset \Gamma$, then $\forall \bg_\epsilon^L\!\in L^2(\OOd)^3\nxs\times\nxs L^2(\OOd)$ such that $\norms{\nxs \Lambda\exs\bg_\epsilon^L-\bPhi_{\text{\tiny L}}  \nxs}_{L^2(\OOd)^3\nxs\times\nxs L^2(\OOd)} \, \leqslant\epsilon$, one has 
\\
$\lim\limits_{\epsilon \rightarrow 0} \norms{\mathcal{S}\bg_\epsilon^L}_{H^{-1/2}(\Gamma)^3 \times H^{-1/2}(\Gamma) \times H^{1/2}(\Gamma)} \,\,=\infty$.
\end{itemize}
\end{theorem}
\begin{proof}
The argument directly follows that of~\cite[Theorem 6.2]{pour2017}.
\end{proof}

\vspace{-3mm}
\subsection{The multiphysics $\mathfrak{L}$ indicator \!\!\!}

Theorem~\ref{TR2} poses two challenges in that:~(i) the featured indicator $ \|\mathcal{S}\bg_\epsilon^L \|_{H^{-1/2}(\Gamma)^3 \times H^{-1/2}(\Gamma) \times H^{1/2}(\Gamma)}$ inherently depends on the unknown support $\Gamma$ of hidden scatterers, and (ii) construction of the wavefront density $\bg_\epsilon^L \in L^2(\OOd)^3\!\times\! L^2(\OOd)$ is implicit in the theorem~\cite{Audibert2014,pour2020(2)}. These are conventionally addressed by replacing $ \|\mathcal{S}\bg_\epsilon^L \|_{H^{-1/2}(\Gamma)^3 \times H^{-1/2}(\Gamma) \times H^{1/2}(\Gamma)}$ with $\norms{\!\bg_\epsilon^L\!\!}_{L^2(\OOd)^3 \times L^2(\OOd)}$ which in turn is computed by way of Tikhonov regularization~\cite{Kress1999} as follows. 

First, note that when the measurements are contaminated by noise (e.g.,~sensing errors and fluctuations in the medium properties), one has to deal with the noisy operator $\Lambda^{\nxs\updelta}\!$ satisfying  
\vspace{-1 mm}
\beq\lb{Ns-op}
\norms{\nxs \Lambda^{\nxs\updelta} -\exs  \Lambda \nxs} \,\,\, \leqslant \,\exs \updelta , 
\vspace{-1 mm}
\eeq
where $\updelta\!>\!0$ is a measure of perturbation in data independent of $\Lambda$. Assuming that $\Lambda^{\nxs\updelta}$ is compact, the Tikhonov-regularized solution $\bg_{\upeta}\nxs$ to (\ref{FF}) is obtained by non-iteratively minimizing the cost functional $J_{\upeta}(\bPhi_{\text{\tiny L}};\,\cdot)\colon \exs L^2(\OOd)^3 \exs \times\! L^2(\OOd) \rightarrow \mathbb{R}$ defined by 
\vspace{-1 mm}  
\beq\label{lssm1}
J_{\upeta}(\bPhi_{\text{\tiny L}};\, \bg) ~ \colon \!\!\! =~ \!\!   \norms{{\Lambda}^{\nxs\updelta} \exs \bg\,-\,\bPhi_{\text{\tiny L}}}^2_{L^2} \,+\,\,\,  \upeta \!\nxs \norms{\bg}^2_{L^2}, \qquad \bg\in L^2(\OOd)^3\nxs\times\nxs L^2(\OOd),
\vspace{-1 mm}  
\eeq
where the regularization parameter $\upeta=\upeta(\updelta, L)$ is obtained by way of the Morozov discrepancy principle~\cite{Kress1999}. On the basis of~\eqref{lssm1}, the well-known linear sampling indicator $\mathfrak{L}$ is constructed as 
\vspace{-1 mm}  
\beq\lb{LSM2}
\mathfrak{L} ~ \colon \!\!\! =~ \!\! \frac{1}{\norms{\bg_{\upeta\nxs}}_{L^2}}, \qquad
\textcolor{black}{
\bg_{\upeta\nxs} ~ \colon \!\!\! =~ \!\! \min_{\bg \exs\in\exs L^2(\OOd)^3\times L^2(\OOd)} J_{\upeta}(\bPhi_{\text{\tiny L}};\, \bg).}
\vspace{-1 mm}  
\eeq

By design, $\mathfrak{L}$ achieves its highest values at the loci of hidden scatterers $\Gamma$. More specifically, the behavior of $\mathfrak{L}$ within $\mathcal{B}$ may be characterized as the following, 
\vspace{-1.5mm}
\beq\lb{LSMB}
\begin{aligned}
& \text{if}\,\,\, L \subset \Gamma \quad \iff \quad  \liminf\limits_{\eta \rightarrow 0} \exs \mathfrak{L}(\bg_{\upeta}) \,>\, 0, \\*[-0.5mm]
& \text{if}\,\,\, L \subset \mathcal{B}\setminus \nxs\overline{\Gamma} \quad \iff \quad \lim\limits_{\eta \rightarrow 0} \mathfrak{L}(\bg_{\upeta}) \,=\, 0. 
\end{aligned}
\vspace*{-1.5mm}
\eeq

\vspace{-3mm}
\subsection{The modified $\mathfrak{G}$ indicator \!\!\!}

In cases where the illumination frequency and/or the background's permeability is sufficiently large so that $\Im{\gamma} \to 0$ according to~\eqref{gam}, one may deduce from~\eqref{up_a}-\eqref{fact} that the second factorization of the poroelastic scattering operator is symmetrized as follows,
\vspace{-1 mm} 
\beq\lb{fact2} 
\Lambda ~=~ {\mathcal{S}}^* \exs T \exs \mathcal{S}. 
\vspace{-1 mm}
\eeq
In this setting, one may deploy the more rigorous $\mathfrak{G}$ indicator~\cite{pour2017} to dispense with the approximations underlying the $\mathfrak{L}$ imaging functional. This results in a more robust reconstruction especially with noisy data~\cite{pour2020(2)}. The $\mathfrak{G}$ indicator takes advantage of the positive and self-adjoint operator $\Lambda_{\sharp}:\, L^2(\OOd)^3 \times L^2(\OOd) \rightarrow L^2(\OOd)^3 \times L^2(\OOd)$ defined on the basis of the scattering operator $\Lambda$ by
\vspace{-1mm}
\beq\lb{Fsd}
\Lambda_{\sharp}\,\colon \!\!\!=\, \frac{1}{2} \big{|}\exs\Lambda+\Lambda^{\nxs *}\big{|} \:+\: \big{|}\frac{1}{2 \textrm{\emph{i}}} (\Lambda\nxs-\Lambda^{\nxs *})\big{|}, 
\vspace{-1mm}
\eeq
with the affiliated factorization~\cite{Kirsch2008} 
 \vspace{-0.5mm}
\beq\lb{facts2}
\Lambda_\sharp ~=~ \mathcal{S}^* \exs T_\sharp \exs \mathcal{S}, 
\vspace{-0.5mm}
\eeq
where in light of Lemma~\ref{T-invs0}, the middle operator $T_\sharp$ is coercive with reference to~\cite[Lemma 5.7]{pour2017} i.e., there exists a constant $c>0$ independent of $\boldsymbol{\uppsi} = \mathcal{S}\bg$ such that $\forall \xxs \boldsymbol{\uppsi}  \in {H^{-1/2}(\Gamma)^3 \times H^{-1/2}(\Gamma) \times H^{1/2}(\Gamma)}$,
\vspace{-0.5mm}
\beq\lb{coT}
(\exs \bg, \exs \Lambda_\sharp \exs  \bg)_{L^2(\OOd)^3 \times L^2(\OOd)}  ~=~ \big \langle \boldsymbol{\uppsi}, \, T_\sharp \xxs \boldsymbol{\uppsi} \big \rangle_{\Gamma} ~\geqslant~ c \norms{\!\boldsymbol{\uppsi} \nxs}^2_{H^{-1/2}(\Gamma)^3 \times H^{-1/2}(\Gamma) \times H^{1/2}(\Gamma)}.
\vspace{-0.5mm}
\eeq

Thanks to~\eqref{coT}, the term $\norms{\!\mathcal{S}\bg_\epsilon\!}^2_{H^{-1/2}(\Gamma)^3 \times H^{-1/2}(\Gamma) \times H^{1/2}(\Gamma)}$ in Theorem~\ref{TR2} may be safely replaced by $(\exs \bg_\epsilon, \exs \Lambda_\sharp \exs  \bg_\epsilon)_{L^2(\OOd)^3 \times L^2(\OOd)}$ which is computable without prior knowledge of~$\Gamma$. 

In presence of noise, the perturbed operator $\Lambda_{\sharp}^{\nxs\updelta}$ is deployed satisfying 
\vspace{-1mm}
\beq\lb{Ns-op2}
\norms{\nxs \Lambda^{\nxs\updelta}_\sharp - \Lambda_\sharp \nxs} \,\,\, \leqslant \,\, \updelta,
\vspace{-1mm}
\eeq
which is assumed to be compact similar to $\Lambda^{\nxs\updelta}$ in~\eqref{Ns-op}. Then, the regularized cost functional $J_{\upalpha,\updelta}(\bPhi_{\text{\tiny L}};\,\cdot)\colon \exs L^2(\OOd)^3 \exs \times L^2(\OOd) \rightarrow \mathbb{R}$ is constructed according to~\cite[Theorems 4.3]{pour2019}, 
\vspace{-1mm}
\beq  \lb{RJ-alph}
J_{\upalpha,\updelta}(\bPhi_{\text{\tiny L}};\, \bg) ~ \colon \!\!\! =~ \!   \norms{\nxs \Lambda^{\nxs\updelta}\bg\,-\,\bPhi_{\text{\tiny L}} \nxs}^2_{L^2} + \,\, \upalpha (\exs \bg, \exs \Lambda^{\nxs\updelta}_\sharp \exs \bg)_{L^2} +\, \upalpha \exs \updelta  \!  \norms{\nxs \bg \nxs}^2_{L^2}, \qquad \bg\in L^2(\OOd)^3\times L^2(\OOd).
\vspace{-1mm}
\eeq

Here, $\upalpha>0$ represents the regularization parameter specified in terms of $\upeta$ from~\eqref{lssm1} as 
\vspace{-1mm}
\beq\lb{Gamm}
\upalpha(\updelta, L) \,\, \colon \!\!\! = \,\, \frac{\upeta(\updelta, L)}{\norms{\Lambda}_{L^2} + \,\, \updelta}.
\vspace{-1mm}
\eeq

In addition, $\updelta>0$ signifies both a measure of perturbation in data and a regularization parameter that, along with $\upalpha$, is designed to create a robust imaging indicator as per Theorems~\ref{GLSM2}. It should be mentioned that $J_{\upalpha,\updelta}$ is convex and that its minimizer $\bg_{\nxs\upalpha,\updelta} \in L^2(\OOd)^3 \times L^2(\OOd)$ solves the linear system 
\vspace{-0.5mm}
\beq \lb{min-RJ} 
\Lambda^{\nxs\updelta *}(\Lambda^{\nxs\updelta} \bg_{\nxs\upalpha,\updelta} -\, \bPhi_{\text{\tiny L}}) ~+~  \upalpha \exs \big( \exs (\Lambda_\sharp^{\nxs\updelta})^{\nxs\frac{1}{2}*} (\Lambda_\sharp^{\nxs\updelta})^{\nxs\frac{1}{2}} \, \bg_{\nxs\upalpha,\updelta} +\, \updelta \, \bg_{\nxs\upalpha,\updelta} \exs \big) ~=~ \bzero,
\vspace{-0.5mm}
\eeq
which can be computed without iterations~\cite{pour2019}. Within this framework, the following theorem rigorously establishes the relation between the range of operator $\bar{\mathcal{S}}^*$ and the norm of penalty term in~\eqref{RJ-alph}. 

\begin{theorem}\lb{GLSM2}
Under the assumptions of Theorem~\ref{TR1} and additional hypothesis that $\Lambda^{\nxs\updelta}$ and $\Lambda_\sharp^{\nxs\updelta}$ are compact, 
\vspace{-4mm}
\begin{multline}\notag
 \bPhi_{\text{\tiny L}} \,\in\, Range(\bar{\mathcal{S}}^*)  ~~\iff~~ 
\Big\{\limsup\limits_{\upalpha \rightarrow 0}\limsup\limits_{\updelta \rightarrow 0}\big(\exs |(\bg_{\nxs\upalpha,\updelta}, \Lambda^{\nxs\updelta}_\sharp \exs\bg_{\nxs\upalpha,\updelta})| \,+\, \updelta \! \norms{\nxs \bg_{\nxs\upalpha,\updelta} \nxs}^2 \nxs \big) \,<\, \infty  \\  
\iff~ \liminf\limits_{\upalpha \rightarrow 0}\liminf\limits_{\updelta \rightarrow 0}\big(\exs |( \bg_{\nxs\upalpha,\updelta}, \Lambda^{\nxs\updelta}_\sharp \exs \bg_{\nxs\upalpha,\updelta})| \,+\, \updelta \! \norms{\nxs \bg_{\nxs\upalpha,\updelta} \nxs}^2 \nxs \big) \,<\, \infty\Big\}, 
\vspace{-1mm}
\end{multline}
where~$\bg_{\nxs\upalpha,\updelta}$ is a minimizer of the perturbed cost functional~\eqref{RJ-alph} in the sense of~\cite[Lemma 6.8]{pour2019}. 
\end{theorem}

In this setting, the imaging indicator $\mathfrak{G}$ takes the form 
\vspace{-1mm}
\beq\lb{GLSMgs}
\mathfrak{G} \,\, = \,\, \dfrac{1}{\sqrt{\norms{\!(\Lambda^{\nxs\updelta}_\sharp)^{\frac{1}{2}} \exs \bg_{\nxs\upalpha,\updelta} \nxs}^2 \exs+\,\, \updelta \nxs \norms{ \bg_{\nxs\upalpha,\updelta} \nxs}^2}},
\vspace{-1mm}
\eeq      
with similar characteristics to $\mathfrak{L}$ as in~\eqref{LSMB} yet more robustness against noise~\cite{pour2020(2)}.   

\begin{rem}
It is worth noting that the sampling-based characterization of $\Gamma$ from near-field data is deeply rooted in geometrical considerations, so that the fracture indicator functionals~(\ref{LSM2}) and~(\ref{GLSMgs}) may exhibit only a minor dependence on the complex contact condition -- described according to~\eqref{GE} by the heterogeneous distribution of:~the stiffness matrix~$\bK(\bxi)$, permeability coefficient~$\varkappa_{\textrm{\tiny f}}(\bxi)$, effective stress coefficient~$\alpha_{\textrm{\tiny f}}(\bxi)$, Skempton coefficient~$\beta_{\textrm{\tiny f}}(\bxi)$, and fluid pressure dissipation factor $\Pi(\bxi)$ on $\bxi \in \Gamma$. This attribute may be traced back to Remark~\ref{LSMrem} where the opening displacement profile $\ba\!\in\!\tilde{H}^{1/2}(L)^3 \nxs\times\nxs \tilde{H}^{1/2}(L)\nxs\times\nxs \tilde{H}^{-1/2}(L)$ -- which is intimately related to the interface law, is deemed arbitrary (within the constraints of admissibility). This quality makes the proposed imaging paradigm particularly attractive in situations where the interfacial parameters are unknown a priori, which opens possibilities for the sequential geometrical reconstruction and interfacial characterization of such anomalies e.g., see~\cite{pour2017(2),pour2018}.   
\end{rem}

\begin{rem}
It should be mentioned that there are recent efforts to systematically adapt the $\mathfrak{G}$ indicator for application to asymmetric scattering operators~\cite{audi2017}. These developments, however, do not lend themselves to poroelastic inverse scattering due to the non-selfadjoint nature of $\Lambda$. A fundamental treatment of the poroelastic $\mathfrak{G}$ indicator in the general case where $\gamma \in \mathbb{C}$ could be an interesting subject for a future study. 
\end{rem}

\vspace{-2mm}
\section{Computational treatment and results} \label{numerics}

This section examines the performance of multiphasic indicators~\eqref{LSM2} and~\eqref{GLSMgs} through a set of numerical experiments. In what follows, the synthetic data are generated within the FreeFem++ computational platform~\cite{Hech2012}.

\vspace*{-2mm}
\subsection{Testing configuration}

Two test setups shown in Figs.~\ref{CN} and~\ref{CT1} are considered where the background is a poroelastic slab ${P}$ of dimensions $22.3$ $\!\times\!$ $22.3$ endowed with (evolving and stationary) hydraulic fracture networks. Following~\cite{ding2013,yew1976}, the properties of Pecos sandstone are used to characterize ${P}$. On setting the reference scales $\mu_r = 5.85$ GPA, $\rho_r = 10^3$ kg/m$^3$, and $\ell_r = 0.14$m for stress, mass density, and length, respectively, the non-dimensionalized quantities of~Table~\ref{prop} are obtained and used for simulations. Accordingly, the complex-valued wave speeds~\cite{shap2015}~affiliated with the modal decomposition in~\eqref{dec} read $\textrm{c}_{\textrm{s}} = 0.66 + 8.8 \!\times\! 10^{-6} \textrm{i}$, $\textrm{c}_{\textrm{p}_1\!} = 1.26 + 3 \!\times\! 10^{-7} \textrm{i}$, and $\textrm{c}_{\textrm{p}_2\!} = 5.8 \!\times\! 10^{-3} + 5.8 \!\times\! 10^{-3} \textrm{i}$. The boundary condition on $\partial P$ is such that the total traction and pore pressure vanish for both incident and scattered fields i.e.,~$(\bn \exs\sip\exs \bC \colon\! \nabla \bu\ff,p\ff) = (\bn \exs\sip\exs \bC \colon\! \nabla \bu,p) = \bzero$. In \emph{Setup I}, a hydraulic fracture network $\Gamma_{1}-\Gamma_{9}$ is induced in four steps as shown in Fig.~\ref{CN}. A detailed description of scatterers including the center $(x_c, y_c)$, length $\ell$, and orientation $\phi$ of each crack $\Gamma_{\kappa}$, $\kappa = \lbrace 1, 2, ..., 9 \rbrace$ is provided in Table~\ref{Num1}. All fractures in this configuration are highly permeable as per Remark~\ref{HPI} and characterized by the interfacial stiffness~$\bK(\bxi) = \bzero$, effective stress coefficient ${\alpha}_{\textrm{\tiny \emph{f}}}(\bxi) = 0.85$, and Skempton coefficient ${\beta}_{\textrm{\tiny \emph{f}}}(\bxi) = 0.3$ on $\bxi \in \!\!{\textstyle \bigcup\limits_{\kappa =1}^{9} \!\! \Gamma_{\!\kappa}}$. The latter quantities are taken from~\cite{naka2007}. Note that in \emph{Setup I}, the excitation and sensing grid $\mathcal{G}$ straddles two (vertical) monitoring wells and the horizontal section of the injection well. Depicted in~Figs.~\ref{CT1} and~\ref{CT2}, \emph{Setup II} features hydraulic fractures $\Gamma_{10}-\Gamma_{15}$ of distinct length scales as described in Table~\ref{Num2}. The discontinuity interfaces in this configuration are modeled as thin poroelastic inclusions characterized by $\lambda_{\text{\sf f}} = 0.1$, $\mu_{\text{\sf f}} = 0.2$, $M_{\text{\sf f}} = 0.33$, $\kappa_{\text{\sf f}} = 5 \times 10^{-7}$, and $\phi_{\text{\sf f}} = 0.35$, while the remaining material parameters are similar to their counterparts in the background as reported in~Table~\ref{prop}. In~\emph{Setup II}, the excitation and measurements are solely conducted in the treatment well shown in~Fig.~\ref{CT1}.     

 \begin{table}[!h]
\vspace*{-1mm}   
 \begin{center}
 \caption{\small Poroelastic properties of the background.} \vspace*{1mm}
\label{prop}
 \begin{tabular}{|l|l|l|}
\hline
{\small{\bf property}}  & {\small{\bf value}} & {\small{\bf dimensionless value}} 
\\ \hline \hline 
{\small{first Lam\'{e} parameter (drained)}}  & $\lambda'=$ 2.74 GPA & $\lambda =$ 0.47 
\\ \hline
{\small{drained shear modulus}}  & $\mu'=$ 5.85  GPA & $\mu =$ 1 
\\ \hline
{\small{Biot modulus}}  & $M'=$ 9.71 GPA & $M =$ 1.66 
\\ \hline
{\small{total density}}  & $\rho'=$  2270 kg/m$^3$ & $\rho =$ 2.27
\\ \hline
{\small{fluid density}}  & $\rho_f'=$  1000 kg/m$^3$ & $\rho_f$ = 1
\\ \hline
{\small{apparent mass density}}  & $\rho_a'=$  117 \, kg/m$^3$ & $\rho_a=$ 0.117
\\ \hline
{\small{permeability coefficient}}  & $\kappa'=$  0.8  mm$^\text{4}$/N & $\kappa=$ 24.5 $\!\times\!$ 10$^{-\text{7}}$
\\ \hline
{\small{porosity}}  & $\phi=$  0.195 & $\phi=$ 0.195
\\ \hline
{\small{Biot effective stress coefficient}}  & $\alpha=$  0.83 & $\alpha=$ 0.83
\\ \hline
{\small{frequency}}  & $\omega'=$  12 kHz& $\omega =$ 3.91
\\ \hline
\end{tabular}
\end{center}
\vspace*{-3.5mm}
\end{table}
 
 \begin{table}[!h]
\vspace*{-1mm}
\begin{center}
\caption{\small Process zone configuration illustrated in Fig.~\ref{CN}:~center $(x_c, y_c)$, length $\ell$, and orientation $\phi$ (with respect to $x$ axis) of cracks $\Gamma_{\kappa}$, $\kappa = \lbrace 1, 2, ..., 9 \rbrace$.} \vspace*{1mm}
\label{Num1}
\begin{tabular}{|c|c|c|c|c|c|c|c|c|c|} \hline
\!\!$\kappa$\!\! & 1 & 2 & 3 & 4&5 & 6 & 7 & 8 & 9 \\ \hline\hline  
\!\!$x_{\text{c}}(\Gamma_\kappa)$\!\!    & \!\!$-5.5$\!\!  & \!\!$-0.25$\!\!  & \!\!$4.3$\!\! & \!\!$-3.3$\!\! & \!\!$1.6$\!\! & \!\!$-4.3$\!\! & \!\!$-3.4$\!\!  & \!\!$-2$\!\! & \!\!$2.9$\!\! \\  
 \hline
 \!\!$y_{\text{c}}(\Gamma_\kappa)$\!\!    & \!\!$0$\!\!  & \!\!$0$\!\!  & \!\!$-1$\!\! & \!\!$0$\!\! & \!\!$0.5$\!\! & \!\!$0.5$\!\! & \!\!$-1.07$\!\!  & \!\!$0$\!\! & \!\!$0$\!\! \\ 
 \hline
\!\!$\ell \exs (\Gamma_\kappa)$\!\!   & \!\!$3$\!\!  & \!\!$2.2$\!\! & \!\!$3$\!\! & \!\!$1$\!\! & \!\!$2.75$\!\! & \!\!$1.8$\!\! & \!\!$1.25$\!\!  &  \!\!$2.4$\!\! & \!\!$2.2$\!\! \\
 \hline
 \!\!$\phi \exs (\Gamma_\kappa)$\!\!   & \!\!$0.47\pi$\!\!  & \!\!$0.6\pi$\!\! & \!\!$0.56\pi$\!\! & \!\!$0.56\pi$\!\! & \!\!$0.42\pi$\!\!  & \!\!$0.5\pi$\!\!  & \!\!$0.37\pi$\!\!  &  \!\!$0.5\pi$\!\! & \!\!$0.56\pi$\!\! \\ 
 \hline
\end{tabular}
\end{center}
\vspace*{-3.5mm}
\end{table}

 \begin{table}[!h]
\vspace*{-1mm}
\begin{center}
\caption{\small Process zone configuration illustrated in~Figs.~\ref{CT1},~\ref{CT2}:~center $(x_c, y_c)$, length $\ell$, and orientation $\phi$ (with respect to $x$ axis) of cracks $\Gamma_{\kappa}$, $\kappa = \lbrace 10, 11, ..., 15 \rbrace$.} \vspace*{1mm}
\label{Num2}
\begin{tabular}{|c|c|c|c|c|c|c|} \hline
\!\!$\kappa$\!\! & 10 & 11 & 12 & 13 & 14 & 15  \\ \hline\hline  
\!\!$x_{\text{c}}(\Gamma_\kappa)$\!\!    & \!\!$-4.76$\!\!  & \!\!$-1.13$\!\!  & \!\!$3.46$\!\! & \!\!$-3.06$\!\! & \!\!$-1.13$\!\! & \!\!$1.16$\!\!  \\  
 \hline
 \!\!$y_{\text{c}}(\Gamma_\kappa)$\!\!    & \!\!$0$\!\!  & \!\!$0$\!\!  & \!\!$0$\!\! & \!\!$0$\!\! & \!\!$0$\!\! & \!\!$0$\!\!  \\ 
 \hline
\!\!$\ell \exs (\Gamma_\kappa)$\!\!   & \!\!$2.35$\!\!  & \!\!$0.74$\!\! & \!\!$7.03$\!\! & \!\!$2.35$\!\! & \!\!$0.74$\!\! & \!\!$7.03$\!\!  \\
 \hline
 \!\!$\phi \exs (\Gamma_\kappa)$\!\!   & \!\!$0.3\pi$\!\!  & \!\!$0.5\pi$\!\! & \!\!$0.46\pi$\!\! & \!\!$0.3\pi$\!\! & \!\!$0.5\pi$\!\!  & \!\!$0.46\pi$\!\!   \\ 
 \hline
\end{tabular}
\end{center}
\vspace*{-3.5mm}
\end{table}

\vspace*{-2mm}
\subsection{Forward scattering simulations}

Every sensing step entails in-plane harmonic excitation via total body forces $\bg_s$ and fluid volumetric sources $g_f$ at a set of points on $\mathcal{G}$. The excitation frequency \mbox{$\omega = 3.91$} is set such that the induced shear wavelength $\lambda_s$ in the specimen is approximately one, serving as a reference length scale. The incident waves interact with the hydraulic fracture network in each setup giving rise to the scattered field $(\bu,p)$, governed by~\eqref{GE}, whose pattern $(\bu\obs,p\obs)$ over the observation grid $\mathcal{G}$ is then computed. For both illumination and sensing purposes, $\mathcal{G}$ is sampled by a uniform grid of $N$ excitation and observation points. In \emph{Setup I}, the H-shaped incident/observation grid is comprised of $N =330$ source/receiver points, while  in \emph{Setup II}, the L-shaped support of injection well is uniformly sampled at $N= 130$ points for excitation and sensing. 

\vspace*{-2mm}
\subsection{Data Inversion}

With the preceding data, one may generate the multiphasic indicator maps affiliated with~\eqref{LSM2} and~\eqref{GLSMgs} in three steps, namely by:~(1) constructing the discrete scattering operators $\boldsymbol{\Lambda}$ and $\boldsymbol{\Lambda}^{\nxs\updelta}$ from synthetic data $(\bu\obs,p\obs)$, (2) computing the trial signature patterns $\bPhi_{\text{\tiny L}}$ pertinent to a poroelastic host domain, and (3) evaluating the imaging indicators ($\mathfrak{L}$ or $\mathfrak{G}$) in the sampling area through careful minimization of the discretized cost functionals~\eqref{lssm1} and~\eqref{RJ-alph} as elucidated in the sequel. 

\subsubsection*{Step 1:~construction of the multiphase scattering operator}
 Given the in-plane nature of wave motion, i.e., that the polarization amplitude of excitation $\bg_s$ and the nontrivial components of associated scattered fields $\bu\obs$ lay in \mbox{the $x-y$} plane of orthonormal bases $(\be_1,\be_2)$, the discretized scattering operator $\boldsymbol{\Lambda}$ may be represented by a $3N\!\times 3N$ matrix with components    
\vspace{0 mm}
\beq\lb{mat2}
\boldsymbol{\Lambda}(4i\nxs+\nxs1\!:\!4i\nxs+\nxs4, \,4j\nxs+\nxs1\!:\!4j\nxs+\nxs4) ~=\, 
\left[\!\begin{array}{cc:c}
u_{11}\! & \!u_{12}\!  & \textsf{u}_{1} \\*[0.75mm]
u_{21}\! & \!u_{22}\!  & \textsf{u}_{2}\\*[1.25mm]
\hdashline
\rule{0pt}{1.25ex}p_{1} & p_{2}  & \!\textsf{p} 
\end{array}\!\right] (\bxi_i,\by_j),  \qquad i,j = 0,\ldots N-1,
\vspace{0 mm}
\eeq
where~${\Lambda}_{rs}(\bxi_i,\by_j)$ $(r,s\!=\!1,2)$ is the $r^{\textrm{th}}$ component of the solid displacement measured at $\bxi_i$ due to a unit force applied at $\by_j$ along the coordinate direction $s$ and $p_{s}$ is the associated pore pressure measurement; also, $\textsf{\bf u}$ signifies the $2\!\times\!1$ solid displacement at $\bxi_i$ due to a unit injection at $\by_j$ and $\textsf{p}$ is the affiliated pore pressure. Note that here the general case is presented where excitations and measurements are conducted along all dimensions. $\boldsymbol{\Lambda}$ may be downscaled as appropriate according to the testing setup.

\vspace{2mm}
\emph{Noisy operator.}~To account for the presence of noise in measurements, we consider the perturbed operator
\vspace{0 mm}
\begin{equation}\label{DFN}
\boldsymbol{\Lambda}^{\nxs\updelta} \,\, \colon \!\!\!= \, (\boldsymbol{I} + \boldsymbol{N}_{\!\epsilon} ) \exs \boldsymbol{\Lambda}, 
\vspace{1 mm}
\end{equation}
where $\boldsymbol{I}$ is the $3N \times 3N$ identity matrix, and $\boldsymbol{N}_{\!\epsilon}$ is the noise matrix of commensurate dimension whose components are uniformly-distributed (complex) random variables in $[-\epsilon, \, \epsilon]^2$. In what follows, the measure of noise in data with reference to definition~\eqref{Ns-op} is $\updelta = \, \norms{\!\boldsymbol{N}_{\!\epsilon} \exs \boldsymbol{\Lambda}\!} \, = 0.05$. 

\subsubsection*{Step 2:~generation of a multiphysics library of scattering patterns}
This step aims to construct a suitable right hand side for the discretized scattering equation in unbounded and bounded domains pertinent to the analytical developments of Section~\ref{SSA} and numerical experiments of this section, respectively. 

\vspace{2mm}
\emph{Unbounded domain in $\R^3$.}
In this case, the poroelastic trial pattern $\bPhi_{\text{\tiny L}} \in L^2(\mathcal{G})^3\times L^2(\mathcal{G})$ is given by Definition~\ref{phi-infinity} indicating that~(a)~the right hand side is not only a function of the dislocation geometry $L$ but also a function of the trial density $\ba\!\in\!\tilde{H}^{1/2}(L)^3 \nxs\times\nxs \tilde{H}^{1/2}(L)\nxs\times\nxs \tilde{H}^{-1/2}(L)$, and (b) computing $\bPhi_{\text{\tiny L}}$ generally requires an integration process at every sampling point $\bx_{\small \circ}$. Conventionally, one may dispense with the integration process by considering a sufficiently localized (trial) density function e.g., see~\cite{pour2017,pour2019}.

\vspace{2mm}
\emph{Bounded domain.} This case corresponds to the numerical experiments of this section where the background is a poroelastic plate ${P}$ of finite dimensions. In this setting, following~\cite{nguy2019}, it is straightforward to show that the trial patterns $\bPhi_{\text{\tiny L}} = (\bu_{\xxs\text{\tiny L}}, p_{\xxs\text{\tiny L}})$ for a finite domain is governed by  
\vspace{-1mm}
\beq\lb{up_a2}
\begin{aligned}
&\nabla \exs\sip\exs (\bC \exs \colon \! \nabla \bu_{\text{\tiny L}}) \,-\, (\alpha-\dfrac{\rho_f}{{\gamma}}) \nabla p_{\xxs\text{\tiny L}}  \,+\, \omega^2 (\rho-\dfrac{\rho_f^2}{{\gamma}}) \bu_{\text{\tiny L}}\,=\, \bzero \quad& \text{in}&\,\,\, {P}\,\backslash \exs \overline{L}, \\*[0.5mm]
&\dfrac{1}{{\gamma}\omega^2}\nabla^2p_{\xxs\text{\tiny L}} \,+\, {M}^{-1} p_{\xxs\text{\tiny L}} \,+\, (\alpha-\dfrac{\rho_f}{{\gamma}}) \nabla \exs\sip\exs \bu_{\text{\tiny L}}  \,=\, 0  \quad & \text{in}&\,\,\,  {P}\,\backslash \exs \overline{L} \\*[1.5mm]
&(\bn \cdot \bC \exs \colon \!  \nabla  \bu_{\text{\tiny L}},  p_{\xxs\text{\tiny L}}) \,=\, \bzero  \quad &\,\text{on}& \,\, \partial P, \\*[3.0mm]
& (\dbr{\bu_{\text{\tiny L}}},\dbr{\exs p_{\xxs\text{\tiny L}}},-\dbr{\xxs q_{\xxs\text{\tiny L}}}) \,=\, \ba, \quad \dbr{\xxs \bt_{\text{\tiny L}}} \,=\, \bzero \quad& \text{on}&\,\,\,  L.
\end{aligned}   
\vspace{0.5mm}
\eeq 

In what follows, the trial signatures $\bPhi_{\bx_{\small \circ}}^{\textrm{\bf{n}},\upiota} = (\textrm{\bf{u}}_{\xxs\bx_{\small \circ}}^{\textrm{\bf{n}},\upiota},\textrm{{p}}_{\bx_{\small \circ}}^{\textrm{\bf{n}},\upiota})(\bxi_i)$ over the observation grid $\bxi_i \nxs\in \mathcal{G}$ are computed separately for every sampling point $\bx_{\small \circ} \in L \subset P$ by solving
\vspace{-1mm}
\beq\lb{PhiL2}
\begin{aligned}
&\nabla \exs\sip\exs (\bC \exs \colon \! \nabla \textrm{\bf{u}}_{\xxs\bx_{\small \circ}}^{\textrm{\bf{n}},\upiota}) \,-\, (\alpha-\dfrac{\rho_f}{{\gamma}}) \nabla \textrm{{p}}_{\bx_{\small \circ}}^{\textrm{\bf{n}},\upiota}  \,+\, \omega^2 (\rho-\dfrac{\rho_f^2}{{\gamma}}) \textrm{\bf{u}}_{\xxs\bx_{\small \circ}}^{\textrm{\bf{n}},\upiota}\,=\, \bzero \quad& \text{in}&\,\,\, {P}\,\backslash \exs \overline{L}, \\*[0.5mm]
&\dfrac{1}{{\gamma}\omega^2}\nabla^2\textrm{{p}}_{\bx_{\small \circ}}^{\textrm{\bf{n}},\upiota} \,+\, {M}^{-1} \textrm{{p}}_{\bx_{\small \circ}}^{\textrm{\bf{n}},\upiota} \,+\, (\alpha-\dfrac{\rho_f}{{\gamma}}) \nabla \exs\sip\exs \textrm{\bf{u}}_{\xxs\bx_{\small \circ}}^{\textrm{\bf{n}},\upiota}  \,=\, - (1-\upiota) \exs \delta (\bx_{\small \circ}\!-\bxi)  \quad & \text{in}&\,\,\,  {P}\,\backslash \exs \overline{L} \\*[1.5mm]
&(\bn \cdot \bC \exs \colon \!  \nabla  \textrm{\bf{u}}_{\xxs\bx_{\small \circ}}^{\textrm{\bf{n}},\upiota},  \textrm{{p}}_{\bx_{\small \circ}}^{\textrm{\bf{n}},\upiota}) \,=\, \bzero  \quad &\,\text{on}& \,\, \partial P, \\*[2.5mm]
& \upiota \exs ( \textrm{\bf{n}} \cdot \bC \exs \colon \!  \nabla  \textrm{\bf{u}}_{\xxs\bx_{\small \circ}}^{\textrm{\bf{n}},\upiota} - |{\sf L}|^{-1} \delta (\bx_{\small \circ}\!-\bxi)\xxs \textrm{\bf{n}},  \textrm{{p}}_{\bx_{\small \circ}}^{\textrm{\bf{n}},\upiota}) \,=\, \bzero \quad& \text{on}&\,\,\,  \bx_{\small \circ}\!+\bR{\sf L},  \\*[2.5mm]
&  \dbr{\exs p_{\bx_{\small \circ}}^{\textrm{\bf{n}},\upiota}} \,=\, 0, \,\,\,\, \dbr{\xxs \textrm{\bf{n}} \cdot \bC \exs \colon \!  \nabla  \textrm{\bf{u}}_{\xxs\bx_{\small \circ}}^{\textrm{\bf{n}},\upiota} \xxs} \,=\, \bzero \quad & \text{on}&\,\,\,  \bx_{\small \circ}\!+\bR{\sf L}. 
\end{aligned}   
\vspace{-1mm}
\eeq 
where $\upiota = 0,1$, and the trial dislocation $L$ is described by an infinitesimal crack $L=\bx_{\small \circ}\!+\bR{\sf L}$ wherein $\bR$ is a unitary rotation matrix, and~{\sf L} represents a vanishing penny-shaped crack of unit normal~$\bn_{\small \circ}\nxs:=\lbrace 1,0,0\rbrace$, so that $\textrm{\bf{n}} = \bR\bn_{\small \circ}$. On recalling~\eqref{mat2}, the three non-trivial components of $(\textrm{\bf{u}}_{\xxs\bx_{\small \circ}}^{\textrm{\bf{n}},\upiota},\textrm{{p}}_{\bx_{\small \circ}}^{\textrm{\bf{n}},\upiota})(\bxi_i)$ in \mbox{the $x-y$} plane, with orthonormal bases $(\be_1,\be_2)$, are arranged into a $3N\!\times\!1$ vector for $\bxi_i \in \mathcal{G}$ as the following
\vspace{-0.5 mm}
\beq\lb{Phi-inf-Dnum}
\bPhi_{\xxs\bx_{\small \circ}}^{\textrm{\bf{n}},\upiota}(3i+1\!:\!3i+3) ~=~ 
\! \left[\!\nxs\begin{array}{c}
\vspace{0.5 mm}
\! \textrm{\bf{u}}_{\xxs\bx_{\small \circ}}^{\textrm{\bf{n}},\upiota} \sip \exs\be_1 \!\! \\*[0.25mm]
\! \textrm{\bf{u}}_{\xxs\bx_{\small \circ}}^{\textrm{\bf{n}},\upiota}  \sip \exs\be_2 \!\! \\*[0.25mm]
\! \textrm{{p}}_{\bx_{\small \circ}}^{\textrm{\bf{n}},\upiota} \!\!
\vspace{0.5 mm}
\end{array} \!\right]\!\! (\bxi_i), \qquad i=0,\ldots N-1.
\vspace{-0.0 mm}
\eeq

\noindent \emph{Sampling.}~With reference to~Figs.~\ref{CN} and~\ref{CT1}, the search area i.e.,~the sampling region is a square $[-5,5]^2\subset P$ probed by a uniform $100 \nxs\times\nxs 100$ grid of sampling points~$\bx_{\small \circ}\!$ where the featured indicator functionals are evaluated, while the unit circle spanning possible orientations for the trial dislocation $L$ is sampled by a grid of $16$ trial normal directions $\textrm{\bf{n}}=\bR\bn_\circ$, and the excitation form -- as a fluid source or an elastic force, is selected via $\upiota$. Accordingly, the multiphasic indicator maps $\mathfrak{L}$ and $\mathfrak{G}$ are constructed through minimizing the cost functionals~\eqref{lssm1} and~\eqref{RJ-alph} for a total of $M = 10000 \nxs\times\nxs 8 \nxs \times\nxs 2$ trial triplets $(\bx_{\small \circ},\textrm{\bf{n}},\iota)$.     

 \begin{rem}
 It is worth mentioning that in the discretized equation
 \vspace{-1 mm}
\beq\lb{Dff}
\boldsymbol{\!\Lambda}^{\nxs\updelta} \, \bg^{\textrm{\bf{n}},\upiota}_{\bx_{\small \circ}}~=~\bPhi_{\xxs\bx_{\small \circ}}^{\textrm{\bf{n}},\upiota},
\vspace{-1 mm}
\eeq  
the scattering operators $\boldsymbol{\Lambda}^{\nxs\updelta}$ is independent of any particular choice of $L(\bx_{\small \circ},\textrm{\bf{n}})$ or $\upiota$.Therefore, one may replace $\bPhi_{\xxs\bx_{\small \circ}}^{\textrm{\bf{n}},\upiota}$ in~\eqref{Dff} with an assembled $3N \!\times\! M$ matrix, encompassing all variations of $L(\bx_{\small \circ},\textrm{\bf{n}})$ and $\upiota$, to solve only one equation for the reconstruction of hidden fractures which is computationally more efficient.       
 \end{rem}
 
\subsubsection*{Step 3:~computation of the multiphysics imaging functionals}\label{RRM}

\emph{$\mathfrak{L}$ map.}~With reference to Theorem~\ref{TR2},~\eqref{Dff} is solved by minimizing the regularized cost functional~\eqref{lssm1}. It is well-known, however, that the Tikhonov functional $J_{\upeta}$ is convex and its minimizer can be obtained without iteration~\cite{Kress1999}. Thus-obtained solution $\bg^{\textrm{\bf{n}},\upiota}_{\bx_{\small \circ}\nxs,\mathfrak{L}}$ to~\eqref{Dff} is a $3N\!\times\! 1$ vector (or $3N\!\times\! M$ matrix for all right-hand sides) identifying the structure of multiphasic wavefront densities on $\mathcal{G}$. On the basis of~\eqref{LSM2} and~\eqref{LSMB}, the multiphysics $\mathfrak{L}$ indicator is then computed as 
\vspace{0 mm}  
\beq\lb{LSMN}
\mathfrak{L}(\bx_{\small \circ}) \,\, := \,\, \frac{1}{\norms{\bg^{\mathfrak{L}}_{\bx_{\small \circ}}\!}_{L^2}}, \qquad
\textcolor{black}{
\bg^{\mathfrak{L}}_{\bx_{\small \circ}}\! ~ \colon \!\!\! =~ \!\! \text{arg\!}\min_{\!\!\!\bg^{\textrm{\bf{n}},\upiota}_{\bx_{\small \circ}\nxs,\mathfrak{L}}} \norms{\bg^{\textrm{\bf{n}},\upiota}_{\bx_{\small \circ},\mathfrak{L}}\nxs}^2_{L^2}.}
\vspace{-1 mm}  
\eeq

\emph{$\mathfrak{G}$ map.}~In the case where $\Im{\gamma} \to 0$, one may construct a more robust approximate solution to~\eqref{Dff} by minimizing $J_{\upalpha,\updelta}$ in~\eqref{RJ-alph}. As indicated in Section~\ref{SSA}, $J_{\upalpha,\updelta}$ is also convex and its minimizer $\bg^{\textrm{\bf{n}},\upiota}_{\bx_{\small \circ},\mathfrak{G}}$ can be computed by solving the discretized linear system~\eqref{min-RJ}. Then, with reference to~(\ref{GLSMgs}), the multiphase indicator $\mathfrak{G}$ is obtained as 
\beq\lb{GLSM-Dgs}
\mathfrak{G}(\bx_{\small \circ}) \,\, := \,\, \dfrac{1}{\sqrt{\norms{\!(\boldsymbol{\Lambda}^{\nxs\updelta}_\sharp)^{\frac{1}{2}} \exs \bg^{\mathfrak{G}}_{\bx_{\small \circ}}\! \nxs}^2 \exs+\,\, \delta \! \norms{\bg^{\mathfrak{G}}_{\bx_{\small \circ}}\! }^2}}, \qquad 
\bg^{\mathfrak{G}}_{\bx_{\small \circ}}\! ~ \colon \!\!\! =~ \!\! \text{arg\!}\min_{\!\!\!\bg^{\textrm{\bf{n}},\upiota}_{\bx_{\small \circ},\mathfrak{G}}} \norms{\bg^{\textrm{\bf{n}},\upiota}_{\bx_{\small \circ},\mathfrak{G}}\nxs}^2_{L^2}.
\eeq

\subsection{Simulation results}\label{CompR}

%

\emph{Near-field tomography of an evolving network}.~With reference to Fig.~\ref{CN}, a hydraulic fracture network ${\textstyle \bigcup\limits_{\kappa =1}^{9} \!\! \Gamma_{\!\kappa}}$ is induced in $P$ in four steps. Following each treatment, the  numerical experiments are conducted as described earlier where the multiphasic scattering patterns $(\bu\obs,p\obs)$ are obtained over an H-shaped sensing grid $\mathcal{G}$ resulting in a $3N \!\times\! 3N = 990 \!\times\! 990$ scattering matrix. The latter is then used to compute the distribution of $\mathfrak{L}$ indicator in the sampling region, at every sensing step, to recover the sequential evolution of the network as shown in the figure.    

\vspace{1.5mm}
\begin{figure}[h!]
\center\includegraphics[width=0.95\linewidth]{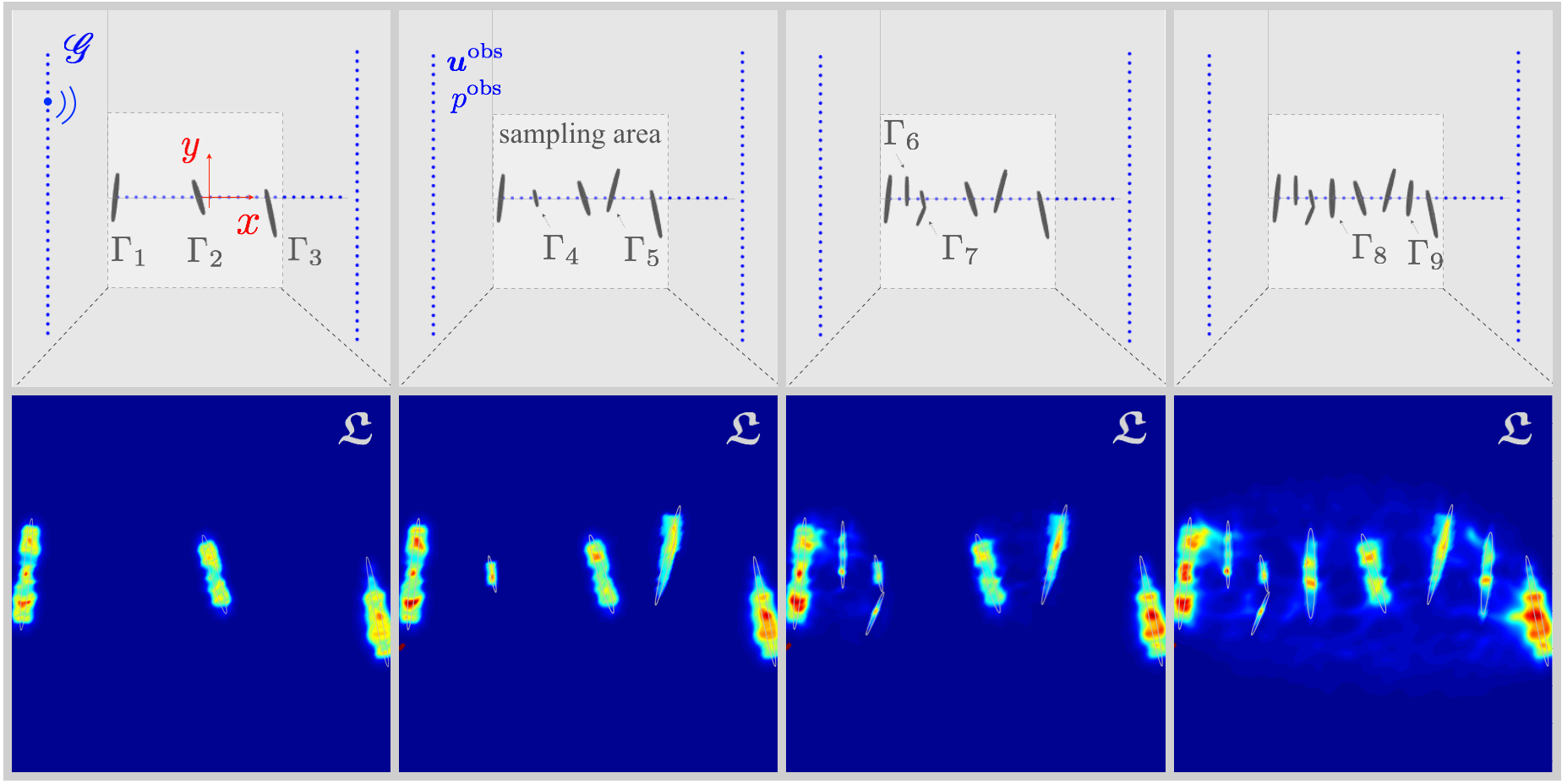} \vspace*{-3mm} 
\caption{Reconstruction of an evolving hydraulic fracture network via the multiphysics $\mathfrak{L}$ indicator:~(top row) sensing and network configurations at each treatment stage, and (bottom row) $\mathfrak{L}$ maps computed via~\eqref{LSMN} on the basis of solid displacements and pore pressures $(\bu\obs,p\obs)$ measured on an H-shaped grid $\bxi_i \in \mathcal{G}, i = \lbrace 1,2, ...,330 \rbrace$ in the injection and two monitoring wells.} \lb{CN}  \vspace*{1mm}
\end{figure} 

\begin{figure}[h!]
\center\includegraphics[width=0.95\linewidth]{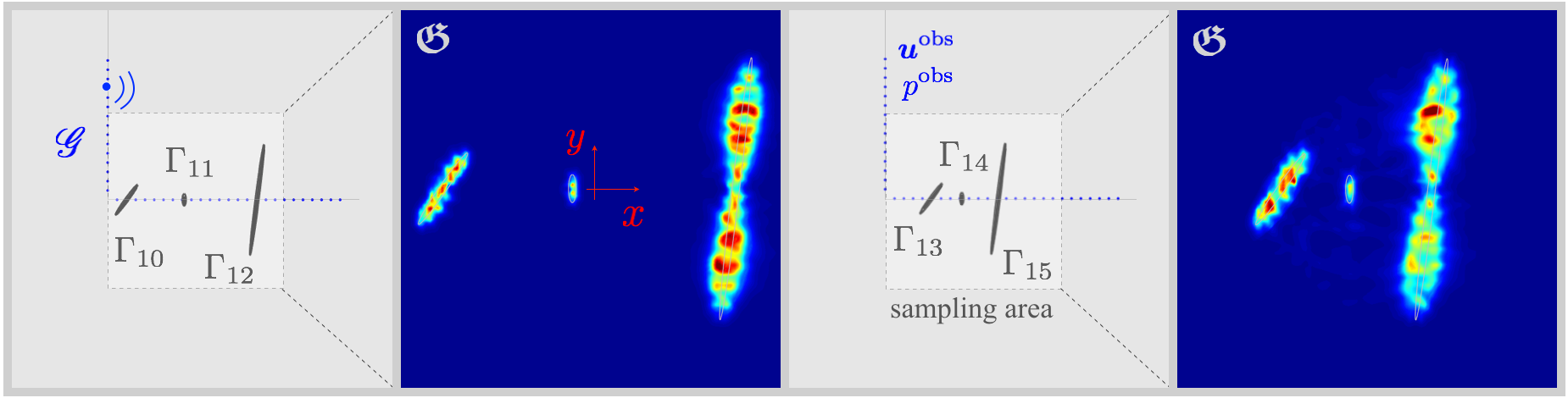} \vspace*{-3mm} 
\caption{Limited-aperture reconstruction via the multiphysics $\mathfrak{G}$ indicator:~(panels 1 and 3) sensing and network configurations, and~(panels 2 and 4)~the affiliated $\mathfrak{G}$ maps constructed via~\eqref{GLSM-Dgs} from the complete dataset i.e.,~solid displacements and pore pressures $(\bu\obs,p\obs)$ observed on an L-shaped grid $\bxi_i \in \mathcal{G}, i = \lbrace 1,2, ...,130 \rbrace$ in the injection well.} \lb{CT1}  \vspace*{1mm}
\end{figure} 

\begin{figure}[h!]
\center\includegraphics[width=0.95\linewidth]{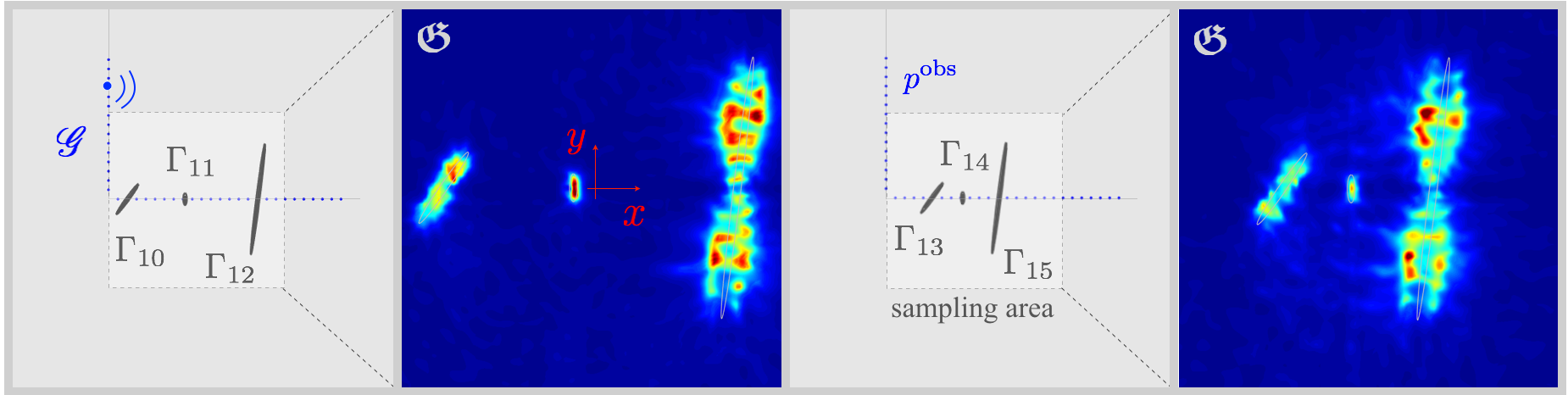} \vspace*{-3mm} 
\caption{Near-field imaging via the pore fluid:~(panels 1 and 3) sensing and network configurations similar to Fig.~\ref{CT1}, and~(panels 2 and 4)~the associated $\mathfrak{G}$ maps computed via~\eqref{GLSM-Dgs} when the excitation at every point $\bxi_i \in \mathcal{G}, i = \lbrace 1,2, ...,130 \rbrace$ is a fluid volumetric source $g_f$ and the measurements are in the form of pore pressure data $p\obs$.} \lb{CT2} \vspace*{-2mm}
\end{figure} 

\emph{Limited-aperture imaging}.~Fig.~\ref{CT1} shows the reconstructed $\mathfrak{G}$ maps via~\eqref{GLSM-Dgs} from the complete dataset, including both solid displacements and pore pressure measurements $(\bu\obs,p\obs)$, on an L-shaped grid $\mathcal{G}$ associated with the injection (or treatment) well, which results in a $3N \!\times\! 3N = 390 \!\times\! 390$ poroelastic scattering matrix. Here, both networks ${\textstyle \bigcup\limits_{\kappa =10}^{12} \!\! \Gamma_{\!\kappa}}$ and ${\textstyle \bigcup\limits_{\kappa =13}^{15} \!\! \Gamma_{\!\kappa}}$ involve hydraulic fractures of different length scales e.g.,~$\Gamma_{11}$ is of $O(1)$ while~$\Gamma_{12}$ is of $O(10)$. One may also observe the interaction between scatterers in the $\mathfrak{G}$ maps when the pair-wise distances between fractures is less than a few wavelengths. It should be mentioned that similar maps are obtained by deploying the $\mathfrak{L}$ imaging functional. However, it is interesting to note that the $\mathfrak{G}$ indicator is successful in reconstructing the fracture network even though $\gamma \in \mathbb{C}$. This may be due to the particular sensor arrangement germane to the hydraulic fracturing application where a subset of the sensing grid (within the injection well) is in fact in close proximity of scatterers $\Gamma_\kappa$, $\kappa = 1,2,...,15$.  

\emph{Acoustic imaging via the pore fluid}.~In hydrofracking, it may be convenient to generate distributed fluid volumetric excitation within the treatment well and simultaneously measure thus-induced pore pressures on the same grid. Data inversion on the basis of acoustic excitation and measurements may expedite reconstruction of the induced fractures. In this spirit,~Fig.~\ref{CT2}  illustrates the reconstruction results for the network and sensing configurations similar to~Fig.~\ref{CT1} where the excitation takes only the form of a fluid volumetric source, and the measurements are pore pressure data $p\obs$. The $\mathfrak{L}$ indicator generates very similar maps.    

%
%

\section{Conclusions} \label{Conc}

A multiphysics data inversion framework is developed for near-field tomography of hydraulic fracture networks in poroelastic backgrounds. This imaging solution is capable of spatiotemporal reconstruction of discontinuous interfaces of arbitrary (and possibly fragmented) support whose poro-mechanical properties are unknown.  The data is processed within an appropriate dimensional framework to allows for simultaneous inversion of elastic and acoustic (i.e., pore pressure) measurements, which are distinct in physics and scale. The proposed imaging indicator is~(a)~inherently non-iterative enabling fast inversion of dense data in support of real-time sensing,~(b)~flexible with respect to the sensing configuration and illumination frequency, and (c) robust against noise in data. Limited-aperture and partial data inversion -- e.g., deep-well tomography as well as imaging based solely on pore pressure excitation and measurements -- are explored. Both $\mathfrak{L}$ and $\mathfrak{G}$ indicators showed success in the numerical experiments germane to the limiting case of high-frequency illumination. It should be mentioned that this imaging modality can be naturally extended to more complex e.g., highly heterogeneous backgrounds. Given the multiphysics nature of data, a remarkable potential would be to deploy this approach within a hybrid data analytic platform to enable not only geometric reconstruction, but also interfacial characterization of discontinuity surfaces, e.g.,~involving the retrieval of permeability profile of hydraulic fractures from poroelastic data.  

\section{Acknowledgements} \label{ackn}

The corresponding author kindly acknowledges the support provided by the National Science Foundation (Grant \#1944812). We would like to mention Rezgar Shakeri's assistance with the preliminary studies on forward scattering by interfaces of infinite permeability. This work utilized resources from the University of  Colorado Boulder Research Computing Group, which is supported by the National Science Foundation (awards ACI-1532235 and ACI-1532236), the University of Colorado Boulder, and Colorado State University.

\appendix

\section{Poroelastodynamics fundamental solution}  \label{fund}

The poroelastodynamics fundamental solution~\cite{scha2012,chen2016} takes the form 
\beq\label{G}
{\text{\bf G}}(\by,\bxi) \,:=\, \left[ \!\! \begin{array}{ll}
\Usfr_{\!\!\! ij} \!\! &\!\! \uffr_{\nxs\! i} \\*[0.1mm]
\exs \psfr_{\!\!\! j} \!&\!\! \pffr \\
\end{array} \!\!\nxs \right](\by,\bxi), \qquad \by,\bxi\in\mathbb{R}^3, \quad \bxi\ne\by, \quad i,j = 1,2,3,
\eeq
where $(\bUsfr,\bpsfr) \in H^1_{\tiny\textrm{loc}}(\R^3 \nxs\setminus\nxs \lbrace\by\rbrace)^{3\times3} \nxs\times\nxs H^1_{\tiny \textrm{loc}}(\R^3 \nxs\setminus\nxs \lbrace\by\rbrace)^{3}$ satisfies
\beq\lb{Ups}
\begin{aligned}
&\nabla \sip (\bC \colon \!\nxs \nabla \bUsfr)(\by,\bxi) - (\alpha-\dfrac{\rho_f}{\gamma}) \nabla \bpsfr(\by,\bxi)  + \omega^2 (\rho-\dfrac{\rho_f^2}{\gamma}) \bUsfr(\by,\bxi)\,=\, - \exs \delta(\by-\bxi) \boldsymbol{I}_{3\times3},    \\*[0.5mm]
&\dfrac{1}{\gamma\omega^2} \nabla^2\bpsfr(\by,\bxi) + {M}^{-1} \bpsfr(\by,\bxi) + (\alpha-\dfrac{\rho_f}{\gamma}) \nabla \sip \bUsfr(\by,\bxi)  \,=\, \bzero,   \quad  \bxi \in \R^3 \nxs\setminus\nxs \lbrace\by\rbrace, \, \bxi \neq \by \in \R^3, 
\end{aligned}   
\eeq  
while $(\buffr,\pffr) \in H^1_{\tiny\textrm{loc}}(\R^3 \nxs\setminus\nxs \lbrace\by\rbrace)^{3} \nxs\times\nxs H^1_{\tiny \textrm{loc}}(\R^3 \nxs\setminus\nxs \lbrace\by\rbrace)$ solves
\beq\lb{upf}
\begin{aligned}
&\nabla \sip (\bC \colon \!\nxs \nabla \buffr)(\by,\bxi) - (\alpha-\dfrac{\rho_f}{\gamma}) \nabla \pffr(\by,\bxi)  + \omega^2 (\rho-\dfrac{\rho_f^2}{\gamma}) \buffr(\by,\bxi)\,=\, \bzero,    \\*[0.5mm]
&\dfrac{1}{\gamma\omega^2} \nabla^2\pffr(\by,\bxi) + {M}^{-1} \pffr(\by,\bxi) + (\alpha-\dfrac{\rho_f}{\gamma}) \nabla \sip \buffr(\by,\bxi)  \,=\, - \exs \delta(\by-\bxi),   \quad  \bxi \in \R^3 \nxs\setminus\nxs \lbrace\by\rbrace, \, \bxi \neq \by \in \R^3, 
\end{aligned}   
\eeq  
both subject to the relevant radiation conditions similar to~\eqref{KS}. On setting
\[
\begin{aligned}
& r \,:=\, |\by-\bxi|, \quad G(k_\varsigma,r) \,:=\, \frac{e^{\text{i}k_\varsigma r}}{4\pi r}, \,\, \varsigma \,=\, \textrm{s}, \textrm{p}_1, \textrm{p}_2, \quad [f]_{,i} \,:=\, \frac{\partial f}{\partial \xi_i}, \,\, i \,=\, 1,2,3, \\*[1mm]
& \qquad \quad \textrm{A}_{1} \,:=\, \frac{k_{\textrm{p}_2\!}^2 (k_{\textrm{p}_1\!}^2 M - \gamma \omega^2)}{\gamma\omega^2(k_{\textrm{p}_1\!}^2-k_{\textrm{p}_2\!}^2 \exs)}, \quad \textrm{A}_{2} \,:=\, \frac{k_{\textrm{p}_1\!}^2 (k_{\textrm{p}_2\!}^2 M - \gamma \omega^2)}{\gamma\omega^2(k_{\textrm{p}_2\!}^2-k_{\textrm{p}_1\!}^2 \exs)},
\end{aligned}
\]
$(\bUsfr,\bpsfr)$ permits the explicit form
\beq\lb{UpsS}
\begin{aligned}
& \Usfr_{\!\!\! ij} \,=\, \dfrac{\gamma}{\omega^2(\rho\gamma-\rho_f^2)} \Big{[} \Big{(} G(k_{\textrm{s}},r) - \textrm{A}_{1} G(k_{\textrm{p}_1\!},r) - \textrm{A}_{2} G(k_{\textrm{p}_2\!},r) \Big{)}_{\!,ij} \,+\, \delta_{ij} k_{\textrm{s}}^2 G(k_{\textrm{s}},r) \Big{]},  \\*[0.5mm]
& \exs \psfr_{\!\!\! j} \,=\, \frac{\omega^2(\alpha\gamma-\rho_f)}{(\lambda+2\mu)(k_{\textrm{p}_1\!}^2-k_{\textrm{p}_2\!}^2 \exs)} \Big{(} G(k_{\textrm{p}_1\!},r) - G(k_{\textrm{p}_2\!},r) \Big{)}_{\!,j}, \quad i,j = 1,2,3,
\end{aligned} 
\eeq
and $(\buffr,\pffr)$ reads 
\beq\lb{upfS}
\begin{aligned}
& \buffr \,=\, - \bpsfr,  \qquad \pffr \,=\, - \frac{\gamma\omega^2}{k_{\textrm{p}_2\!}^2-k_{\textrm{p}_1\!}^2}\Bigg[ \Big{(} k_{\textrm{p}_1\!}^2 - \frac{\omega^2 (\rho \gamma -\rho_f^2)}{\gamma(\lambda+2\mu)} \Big{)} G(k_{\textrm{p}_1\!},r) - \Big{(} k_{\textrm{p}_2\!}^2 - \frac{\omega^2 (\rho \gamma -\rho_f^2)}{\gamma(\lambda+2\mu)} \Big{)} G(k_{\textrm{p}_2\!},r) \Bigg], 
\end{aligned} 
\eeq
In this setting, the associated fundamental tractions $\bTsfr$ and $\btffr$ on a surface $\Gamma \subset \R^3 \nxs\setminus\nxs \lbrace\by\rbrace$ characterized by the unit normal vector $\bn$ may be specified as
\beq\lb{Ts}
\begin{aligned}
& \Tsfr_{\!\!\! ij}(\by,\bxi) \,=\, n_{k}(\bxi) C_{ik\ell\varsigma}  \Usfr_{\!\!\! \ell j,\varsigma}(\by,\bxi) - \alpha n_i(\bxi) \psfr_{\!\!\! j}(\by,\bxi),   \\*[1mm]
& \exs \tffr_{\! i}(\by,\bxi) \,=\, n_{k}(\bxi) C_{ik\ell\varsigma}  \uffr_{\nxs\!  \ell,\varsigma} (\by,\bxi) - \alpha n_i(\bxi) \pffr(\by,\bxi),  \quad\,\,  \bxi \in \Gamma, \,  \by \in \R^3 \nxs\setminus\nxs \overline{\Gamma}, 
\end{aligned} 
\eeq
where Einstein's summation convention applies over the repeated indexes. In addition, the normal relative fluid-solid displacements $\bqsfr$ and $\qffr$ across $\Gamma$ may be expressed as  
\beq\lb{Qs}
\begin{aligned}
& \qsfr_{\!\!\! j}(\by,\bxi) \,=\, \dfrac{1}{\gamma\omega^2}\Big{(}n_i(\bxi) \psfr_{\!\!\! j,i}(\by,\bxi) \,-\, \rho_f \omega^2 n_i(\bxi)\Usfr_{\!\!\! ij}(\by,\bxi)\Big{)},   \\*[1mm]
& \exs \qffr(\by,\bxi) \,=\, \dfrac{1}{\gamma\omega^2}\Big{(}n_i(\bxi) \pffr_{\!\!\! ,i}(\by,\bxi) \,-\, \rho_f \omega^2 n_i(\bxi)\uffr_{\! i}(\by,\bxi)\Big{)},  \quad\,\,  \bxi \in \Gamma, \,  \by \in \R^3 \nxs\setminus\nxs \overline{\Gamma}.
\end{aligned} 
\eeq

\section{Wellposedness of the poroelastic scattering problem}\label{Wellp}

Given $(\bt\ff,q\ff,p\ff) \in H^{-1/2}(\Gamma)^3 \times H^{-1/2}(\Gamma) \times H^{1/2}(\Gamma)$ and the interface operator $\mathfrak{P}: \tilde{H}^{{1}/{2}}(\Gamma)^3\times \tilde{H}^{{1}/{2}}(\Gamma) \times \tilde{H}^{-{1}/{2}}(\Gamma) \rightarrow {H}^{-{1}/{2}}(\Gamma)^3\times {H}^{-{1}/{2}}(\Gamma) \times {H}^{{1}/{2}}(\Gamma)$ satisfying~\eqref{ImP}, then the direct scattering problem~\eqref{GE}-\eqref{KS} has a unique solution that continuously depends on $(\bt\ff,q\ff,p\ff)$. Note that $q\ff$ is continuously dependent on $(\bu\ff \sip \bn,\nabla p\ff \sip \bn)$ according to~\eqref{Neu}. First, observe that~\eqref{GE}-\eqref{KS} may be expressed on $\mathcal{B}\setminus\nxs\overline{\Gamma}$ in the variational form in terms of $(\bu,p) \in H^1_{\tiny\textrm{loc}}({\R^3}\nxs\setminus\nxs \overline{\Gamma})^3 \nxs\times\nxs H^1_{\tiny\textrm{loc}}({\R^3}\nxs\setminus\nxs \overline{\Gamma})$ so that $\forall (\texttt{\bf u}',\mathtt{p}')  \in H^1_{\tiny\textrm{loc}}({\R^3}\nxs\setminus\nxs \overline{\Gamma})^3 \nxs\times\nxs H^1_{\tiny\textrm{loc}}({\R^3}\nxs\setminus\nxs \overline{\Gamma})$ satisfying the radiation condition~\eqref{KS},
\beq\lb{Wik-GE}
\begin{aligned}
&\Big\langle (\bt\ff + \tilde{\alpha}_{\textrm{\tiny f}} \exs p\ff \exs \bn, \, q\ff, \, p\ff), (\dbr{{\butt}'}, \dbr{{\ptt}'}, \vartheta_{\textrm{\tiny f}} \llangle{\xxs {\ptt}'}\rrangle) \Big\rangle_{\Gamma}  \,=\, \\*[2.75 mm]
&\int_{\mathcal{B}\setminus\nxs\Gamma} \big\lbrace \nabla \bar{\butt}' \exs \colon \! \bC \exs \colon \! \nabla \bu \,+\, \frac{1}{\gamma\omega^2} \nabla\bar{\ptt}' \sip \nabla p    \,-\, (\alpha - \frac{\rho_f}{\gamma}) \exs (\nabla \sip \bar{\butt}' \xxs p  \exs + \exs  \bar{\ptt}' \nabla \sip \bu)  \big\rbrace \, \text{d}V_{\bxi} ~-  \\*[1.0 mm]
& \int_{\mathcal{B}\setminus\nxs\Gamma} \Big\lbrace \omega^2 (\rho - \frac{\rho_f^2}{\gamma}) \exs \bar{\butt}' \nxs\sip\exs \bu + M^{-1} \bar{\ptt}' p \Big\rbrace  \, \text{d}V_{\bxi} \,+\, \frac{\rho_f}{\gamma} \! \int_{\Gamma}  \Big( \llangle{\xxs p }\rrangle \dbr{\xxs\bar{\butt}'}   + \dbr{\xxs p \xxs} \llangle{\xxs \bar{\butt}' }\rrangle   +  \llangle{\xxs \bar{\ptt}' }\rrangle \dbr{\bu}  + \dbr{\xxs\bar{\ptt}'} \llangle{\bu }\rrangle   \Big) \sip \bn \, \text{d}S_{\bxi} \,\,-  \\*[1.875 mm]
&  \Big\langle \big(-\dbr{\bu} \sip \bK \nxs+ \tilde{\alpha}_{\textrm{\tiny f}} \exs  \llangle{\xxs p }\rrangle \bn, - \dfrac{\varkappa_{\textrm{\tiny f}}}{\textrm{i}\omega\Pi} \dbr{\xxs p}, \llangle{\xxs p }\rrangle + \frac{{\beta}_{\textrm{\tiny f}}}{1-  \tilde{\alpha}_{\textrm{\tiny f}}{\beta}_{\textrm{\tiny f}}} \bn \sip \bK \dbr{\bu}\big), \big(\dbr{{\butt}'}, \dbr{{\ptt}'}, \vartheta_{\textrm{\tiny f}} \llangle{\xxs {\ptt}'}\rrangle\big) \Big\rangle_{\Gamma} \,\,- \\*[2.75 mm]
&  \int_{\partial \mathcal{B}} \Big\lbrace \bar{\butt}' \sip \exs \bt \,+\, \frac{1}{\gamma \omega^2} \bar{\ptt}' \exs \nabla p \sip \bn \,+\, \frac{\rho_f}{\gamma}  \bar{\butt}' \sip \bn \exs p \Big\rbrace  \, \text{d}S_{\bxi},
\end{aligned}
\eeq
where $\vartheta_{\textrm{\tiny f}} \in L^{\infty}(\Gamma)$ is defined by
\[
{\vartheta}_{\textrm{\tiny f}} ~:=~ \overline{\frac{\Pi \alpha_{\textrm{\tiny f}}}{\textrm{k}_n {\beta}_{\textrm{\tiny f}}} ({1-  \tilde{\alpha}_{\textrm{\tiny f}}{\beta}_{\textrm{\tiny f}}})} \qquad \text{on} \,\,\,\, \Gamma,
\]
with the `bar' indicating complex conjugate. Note that in light of~\eqref{KS}, one has 
\[
\lim\limits_{R\rightarrow \infty} \int_{\partial \mathcal{B}} \exs \Big\lbrace \bar{\butt}' \sip \exs \bt \,+\, \frac{1}{\gamma \omega^2} \bar{\ptt}' \exs \nabla p \sip \bn \,+\, \frac{\rho_f}{\gamma}  \bar{\butt}' \sip \bn \exs p \Big\rbrace  \, \text{d}S_{\bxi} ~=~ 0.
\]
It is worth mentioning that~\eqref{Wik-GE} is obtained by premultiplying the first of~\eqref{GE} by $\bar{\texttt{\bf u}}'$ and its second by $\bar{\mathtt{p}}'$; the results are then added and integrated by parts over $\mathcal{B}\setminus\nxs\overline{\Gamma}$. The duality pairing $\dualGA{\cdot}{\cdot}$ on the left hand side of~\eqref{Wik-GE} is continuous with respect to $\bt\ff$, $q\ff$, and $p\ff$ owing to the continuity of the trace mapping $({\butt}', {\ptt}') \to (\dbr{{\butt}'},\dbr{\xxs{\ptt}'}, \llangle \xxs \ptt' \txs \rrangle)$ from $H^1_{\tiny\textrm{loc}}({\R^3}\nxs\setminus\nxs \overline{\Gamma})^3 \nxs\times\nxs H^1_{\tiny\textrm{loc}}({\R^3}\nxs\setminus\nxs \overline{\Gamma})$ into $\tilde{H}^{1/2}(\Gamma)^3 \times \tilde{H}^{1/2}(\Gamma) \times {H}^{1/2}(\Gamma)$. Now let us define $\forall (\texttt{\bf u}',\mathtt{p}')  \in H^1_{\tiny\textrm{loc}}({\R^3}\nxs\setminus\nxs \overline{\Gamma})^3 \nxs\times\nxs H^1_{\tiny\textrm{loc}}({\R^3}\nxs\setminus\nxs \overline{\Gamma})$ subject to~\eqref{KS}, the coercive operator 
\beq\lb{cv}
\begin{aligned} 
& \textrm{A}\big((\bu,p),({\butt}', {\ptt}')\big) ~:=\, \int_{\mathcal{B}\setminus\nxs\Gamma} \big\lbrace \nabla \bar{\butt}' \exs \colon \! \bC \exs \colon \! \nabla \bu \,+\, \frac{1}{\gamma\omega^2} \nabla\bar{\ptt}' \sip \nabla p    \,\,- \\*[1 mm] 
&\hspace{4.2cm} (\alpha - \frac{\rho_f}{\gamma}) \exs (\nabla \sip \bar{\butt}' \xxs p  \exs + \exs \bar{\ptt}' \nabla \sip \bu) \,+\, \bar{\ptt}' p \,+\,   \frac{\rho_f^2}{\gamma} \exs \omega^2 \exs \bar{\butt}' \nxs\sip\exs \bu   \big\rbrace \, \text{d}V_{\bxi}.
\end{aligned}
\eeq
On invoking $\mu_r = \lambda'$ i.e., $\lambda = 1$ from Section~\ref{DiP} and recalling from~\eqref{gam} that $|\alpha - {\rho_f}/{\gamma}|<1$, one may observe 
\beq\lb{cv2}
\begin{aligned} 
&\mathfrak{R}(\textrm{A})((\bu,p),(\bu,p)) \,\geqslant\,\,  \int_{\mathcal{B}\setminus\nxs\Gamma} \big\lbrace \big(\lambda- |\alpha - \frac{\rho_f}{\gamma}|\big)\! \norms{\nxs \nabla \sip \bu \nxs}^2  +\,\, \mu \norms{\nxs\nabla \bu\nxs}^2  +\,\, \frac{\mathfrak{R}\gamma}{\gamma^2\omega^2}  \norms{\nxs\nabla p\nxs}^2 \,+\,  \\*[1 mm] 
& \hspace{2.6cm} \big(1- |\alpha - \frac{\rho_f}{\gamma}|\big) \norms{p}^2 +\,\, \frac{\rho_f^2}{\gamma^2} \exs \omega^2 \mathfrak{R}\gamma \nxs\norms{\nxs\bu\nxs}^2 \!\big\rbrace \, \text{d}V_{\bxi} \,\geqslant\,\, \textrm{c} \exs \big(\!\norms{\nxs\bu\nxs}^2_{H^1(\mathcal{B}\setminus\nxs \Gamma)^3} + \norms{p}^2_{H^1(\mathcal{B}\setminus\nxs \Gamma)}\!\big),
\end{aligned}
\eeq
for a constant $\textrm{c}>0$ independent of $(\bu,p)$. Next, define $\forall (\texttt{\bf u}',\mathtt{p}')  \in H^1_{\tiny\textrm{loc}}({\R^3}\nxs\setminus\nxs \overline{\Gamma})^3 \nxs\times\nxs H^1_{\tiny\textrm{loc}}({\R^3}\nxs\setminus\nxs \overline{\Gamma})$ subject to~\eqref{KS}, the compact operator
\beq\lb{compact}
\begin{aligned} 
& \textrm{B}\big((\bu,p),({\butt}', {\ptt}')\big) ~=\, - \big\langle \big(-\dbr{\bu} \sip \bK \nxs+ \tilde{\alpha}_{\textrm{\tiny f}} \exs  \llangle{\xxs p }\rrangle \bn, - \dfrac{\varkappa_{\textrm{\tiny f}}}{\textrm{i}\omega\Pi} \dbr{\xxs p}, \llangle{\xxs p }\rrangle + \frac{{\beta}_{\textrm{\tiny f}}}{1-  \tilde{\alpha}_{\textrm{\tiny f}}{\beta}_{\textrm{\tiny f}}} \bn \sip \bK \dbr{\bu}\big), \big(\dbr{{\butt}'}, \dbr{{\ptt}'}, \vartheta_{\textrm{\tiny f}} \llangle{\xxs {\ptt}'}\rrangle\big) \big\rangle_{\Gamma} \,\, - \\*[2 mm]
&\qquad\quad  \int_{\mathcal{B}\setminus\Gamma} \big\lbrace \rho \exs \omega^2  \bar{\butt}' \nxs\sip\exs \bu \,+\, (M^{-1}+1) \xxs \bar{\ptt}' p \big\rbrace  \, \text{d}V_{\bxi} \,+\, \frac{\rho_f}{\gamma} \! \int_{\Gamma}  \big( \llangle{\xxs p }\rrangle \dbr{\xxs\bar{\butt}'}   + \dbr{\xxs p \xxs} \llangle{\xxs \bar{\butt}' }\rrangle   +  \llangle{\xxs \bar{\ptt}' }\rrangle \dbr{\bu}  + \dbr{\xxs\bar{\ptt}'} \llangle{\bu }\rrangle   \big) \sip \bn \, \text{d}S_{\bxi}, 
\end{aligned}
\eeq
where
\[
\begin{aligned} 
& \big| \textrm{B}\big((\bu,p),({\butt}', {\ptt}')\big) \big| \leqslant\, \text{c} \exs \Big\lbrace \nxs
\norms{\!\bu\!}_{L^2(\mathcal{B}\setminus\Gamma)^3} \, \norms{\nxs\butt'\!}_{L^2(\mathcal{B}\setminus\Gamma)^3} \nxs+ \norms{\nxs p\nxs}_{L^2(\mathcal{B}\setminus\Gamma)} \, \norms{\nxs\ptt'\!}_{L^2(\mathcal{B}\setminus\Gamma)} +  \\*[1 mm]
&\quad\, \norms{\!\dbr{\bu}\!}_{L^2(\Gamma)^3} \, \norms{\!\dbr{\butt'}\!}_{L^2(\Gamma)^3} \nxs+ \norms{\! \dbr{\exs p \exs}\!}_{L^2(\Gamma)} \, \norms{\!\dbr{\xxs \ptt'}\!}_{L^2(\Gamma)} \nxs+ \norms{\! \llangle{\exs p \xxs}\rrangle\!}_{L^2(\Gamma)} \, \norms{\!\llangle{\xxs \ptt'}\rrangle\!}_{L^2(\Gamma)} \nxs+ \\*[2 mm]
&\quad\,  \norms{\! \llangle{\xxs p }\rrangle\!}_{L^2(\Gamma)} \exs \norms{\!\dbr{\butt'}\!}_{L^2(\Gamma)^3} \nxs+ \norms{\! \dbr{\xxs p }\!}_{L^2(\Gamma)} \exs \norms{\!\llangle{\xxs \bar{\butt}' }\rrangle\!}_{L^2(\Gamma)^3} \nxs+ \norms{\!\dbr{\bu}\!}_{L^2(\Gamma)^3} \exs \norms{\! \llangle{\xxs \ptt' }\rrangle\!}_{L^2(\Gamma)} \nxs+ \norms{\!\llangle{\xxs \bu }\rrangle\!}_{L^2(\Gamma)^3} \exs \norms{\! \dbr{\xxs \ptt' }\!}_{L^2(\Gamma)}  \!   
\! \Big\rbrace, 
\end{aligned}
\]
for a constant $\text{c}$ independent of $(\bu,p)$ and $({\butt}', {\ptt}')$. Then, the compact embedding of $H^1(\mathcal{B}\setminus\nxs\overline{\Gamma})$ into $L^2(\mathcal{B}\setminus\nxs\overline{\Gamma})$ and the compactness of the trace operator indicates that $\text{B}$ defines a compact perturbation of $\textrm{A}\big((\bu,p),({\butt}', {\ptt}')\big)$.
Therefore, on letting $R\to\infty$, one may observe that the problem \eqref{Wik-GE} is of Fredholm type, and thus, is well-posed as soon as the uniqueness of a solution is guaranteed. Assume that $(\bt\ff, q\ff, p\ff) = \bzero$. Then, on invoking~\eqref{gam} and~\eqref{frakP}, one may show that when $R\to\infty$, 
\beq\lb{ImWik-GE}
-\frac{\omega}{\kappa} \nxs \norms{\nxs \bq \nxs}^2_{L^2(\R^3 \setminus\nxs\Gamma)^3} +~ \Im \exs \big\langle \mathfrak{P}\big(\dbr{\bu},\dbr{\exs {p} \xxs}, - \dbr{q} \big), \big(\dbr{\bu},\dbr{\exs {p} \xxs}, - \dbr{q} \big) \big\rangle_{\Gamma}  \,=\, 0, 
\eeq
where
\beq\lb{qp}
\bq \,\,:=\,\, \dfrac{1}{\gamma\omega^2}\big(\nabla p \,- \rho_f \omega^2 \bu \xxs \big) \qquad \text{in}\,\,\,  {\R^3}\nxs\setminus\nxs\overline{\Gamma}.
\eeq
It should be mentioned that~\eqref{ImWik-GE} is obtained by (a)~premultiplying the first of~\eqref{GE} by $\bar{\bu} \in H^1_{\tiny\textrm{loc}}({\R^3}\nxs\setminus\nxs \overline{\Gamma})^3$,~(b) conjugating the second of~\eqref{GE} and multiplying the results by $-{p} \in H^1_{\tiny\textrm{loc}}({\R^3}\nxs\setminus\nxs \overline{\Gamma})$, and~(c) integrating by parts over ${\mathcal{B}}\nxs\setminus\nxs\overline{\Gamma}$ and adding the results. On letting $R\to\infty$, one finds~\eqref{ImWik-GE} indicating that $\bq = \bzero$ in $\R^3 \nxs\setminus\nxs\overline{\Gamma}$ by the premise of Remark~\ref{R_Wellp}. As a result, the first equation in~\eqref{GE} will be reduced to the Navier equation that is subject to the (exponentially decaying) poroelastodynamics radiation conditions~\eqref{KS} which implies that $\bu = \bzero$ in $\R^3\nxs\setminus\nxs\overline{\Gamma}$. The second equation in~\eqref{GE} then reads $p = 0$ in $\R^3\nxs\setminus\nxs\overline{\Gamma}$ which completes the proof for the uniqueness of a scattering solution, and thus, substantiates the wellposedness of the forward problem.

\section{Proof of Lemma~\ref{H*}}  \label{H*pruf}

Let us define
\beq\lb{SigZet}\nonumber
\begin{aligned}
&\boldsymbol{\sigma}^{\iota} ~:=~ \bC \exs \colon \! \nabla \bu^{\iota} \,-\, \alpha \exs p^{\iota} \bI_{3\times 3},  \quad\,\,\,\,\,  \zeta^{\iota} ~:=~ M^{-1} p^{\iota} \,+\, \alpha \nabla \exs\sip\exs \bu^{\iota} \,\,\,  &\text{in}\,\,\,  {\R^3}\nxs\setminus\nxs\overline{\Gamma},\\*[1 mm]
& \boldsymbol{\sigma}_{\nxs\ba} ~:=~ \bC \exs \colon \! \nabla \bu_{\ba} \,-\, \alpha \exs p_{\ba} \bI_{3\times 3}, \quad \zeta_{\exs\ba} ~:=~ M^{-1} p_{\ba} \,+\, \alpha \nabla \exs\sip\exs \bu_{\ba}   \,\,\,  &\text{in}\,\,\,  {\R^3}\nxs\setminus\nxs\overline{\Gamma},  
\end{aligned}
\eeq
where $(\bu\ff,p\ff) \in H^1_{\tiny\textrm{loc}}({\R^3}\nxs\setminus\nxs\overline{\mathcal{G}})^3 \nxs\times\nxs H^1_{\tiny\textrm{loc}}({\R^3}\nxs\setminus\nxs\overline{\mathcal{G}})$ and $(\bu_{\ba},p_{\ba}) \in H^1_{\tiny\textrm{loc}}({\R^3}\nxs\setminus\nxs \overline{\Gamma})^3 \nxs\times\nxs H^1_{\tiny\textrm{loc}}({\R^3}\nxs\setminus\nxs \overline{\Gamma})$ solve~\eqref{uf} and~\eqref{up_a} respectively. Then, the poroelastic reciprocal theorem~\cite{scha2012} reads 
\beq\lb{recip}
\nabla \bar{\bu}_{\ba} \colon \nxs \boldsymbol{\sigma}^{\iota}\nxs \,+\, \bar{\zeta}_{\exs\ba} \exs p\ff ~=~ \bar{\boldsymbol{\sigma}}_{\ba} \colon \! \nabla \bu^{\iota}\nxs \,+\, \bar{p}_{\ba} \exs \zeta^{\iota}. 
\eeq
Now, on setting 
\beq\lb{Vq}
\bq\ff \,\,:=\,\, \dfrac{1}{\gamma\omega^2}\big(\nabla p\ff \,-\, \rho_f \omega^2 \bu\ff\big), \quad \bq_{\ba} \,\,:=\,\, \dfrac{1}{\bar{\gamma}\omega^2}\big(\nabla p_{\ba} \,-\, \rho_f \omega^2 \bu_{\ba}\big) \quad \text{in}\,\,\,  {\R^3}\nxs\setminus\nxs\overline{\Gamma}, 
\eeq
observe that
\[
\begin{aligned}
&\int_{\R^3\setminus\nxs\Gamma} \big\lbrace \nabla \bar{\bu}_{\ba} \colon \nxs \boldsymbol{\sigma}^{\iota}\nxs \,+\, \bar{\zeta}_{\exs\ba} \exs p\ff \big\rbrace \, \text{d}V_{\bxi} ~:=~ \! - \int_{\Gamma} \big\lbrace \dbr{\bar{\bu}_{\ba}} \nxs \sip \exs \bt\ff \,-\, \dbr{\exs\bar{q}_{\ba}} \exs p\ff \big\rbrace\, \text{d}S_{\bxi} \,+ \int_{\mathcal{G}} \bar{\bu}_{\ba} \sip \bg_s \, \text{d}S_{\bxi} \,\,+ \\*[1 mm]
& \quad \qquad \qquad \rho\omega^2 \!\! \int_{\R^3\setminus\nxs\Gamma} \bar{\bu}_{\ba} \nxs\sip\exs \bu\ff  \, \text{d}V_{\bxi} \,+\, \rho_f\omega^2 \!\! \int_{\R^3\setminus\nxs\Gamma} \big\lbrace \bar{\bu}_{\ba} \nxs\sip\exs \bq\ff \,+\, \bar{\bq}_{\ba} \nxs\sip\exs \bu\ff \big\rbrace \, \text{d}V_{\bxi} \,+\, \gamma\omega^2 \!\! \int_{\R^3\setminus\nxs\Gamma} \bar{\bq}_{\ba} \nxs\sip\exs \bq\ff  \, \text{d}V_{\bxi}, \\*[2 mm]
&\int_{\R^3\setminus\nxs\Gamma} \big\lbrace \bar{\boldsymbol{\sigma}}_{\ba} \colon \! \nabla \bu^{\iota}\nxs \,+\, \bar{p}_{\ba} \exs \zeta^{\iota} \big\rbrace \, \text{d}V_{\bxi} ~:=~ \! \int_{\Gamma}  \dbr{\exs \bar{p}_{\ba}} \exs q\ff \, \text{d}S_{\bxi} \,- \int_{\mathcal{G}} \bar{p}_{\ba} \sip \bg_f \, \text{d}S_{\bxi} \,\,+ \\*[1 mm]
& \quad \qquad \qquad \rho\omega^2 \!\! \int_{\R^3\setminus\nxs\Gamma} \bar{\bu}_{\ba} \nxs\sip\exs \bu\ff  \, \text{d}V_{\bxi} \,+\, \rho_f\omega^2 \!\! \int_{\R^3\setminus\nxs\Gamma} \big\lbrace \bar{\bu}_{\ba} \nxs\sip\exs \bq\ff \,+\, \bar{\bq}_{\ba} \nxs\sip\exs \bu\ff \big\rbrace \, \text{d}V_{\bxi} \,+\, \gamma\omega^2 \!\! \int_{\R^3\setminus\nxs\Gamma} \bar{\bq}_{\ba} \nxs\sip\exs \bq\ff  \, \text{d}V_{\bxi}.
\end{aligned}
\]

Then, by invoking~\eqref{oH},~\eqref{Hstar}, and~\eqref{up_a}, it is evident that
\[
\begin{aligned}
& \big\langle \mathcal{S}(\bg),\ba \big\rangle_{\Gamma} ~=~ \big(\bg,\mathcal{S}^*(\ba)\big)_{L^2(\mathcal{G})},
 \end{aligned}
\]
which completes the proof.

\section{Proof of Lemma~\ref{H*p}}  \label{S*pruf}

Observe that $\mathcal{S}^*$ of Lemma~\ref{H*} possesses the integral form
\beq\lb{Dlp}
\begin{aligned}
\mathcal{S}^*(\ba) ~=~ \int_{\Gamma}   \bar{\bSig}(\bxi,\by) \cdot \ba(\by)  \, \text{d}S_{\by}, \qquad \bSig \,=\, \left[ \!\! \begin{array}{lll}
\\*[-5.85mm]
\bTsfr \!\! &\!\! \bqsfr \!\! &\! \bpsfr \\*[0mm]
 \btffr \!\! &\!\! \qffr \!\! &\!  \pffr
\end{array} \!\!\nxs \right], \qquad  \bxi \in \mathcal{G}.
\end{aligned}
\eeq

Note that the kernel ${\bSig}(\bxi,\by)$ is derived from the poroelastodynamics fundamental solution of~\ref{fund} which also appeared in the integral representation~\eqref{vinf2}. In this setting, the smoothness of $\bar{\bSig}$ substantiates the compactness of $\mathcal{S}^*$ as an application from $\tilde{H}^{1/2}(\Gamma)^3 \nxs\times\nxs \tilde{H}^{1/2}(\Gamma) \nxs\times\nxs \tilde{H}^{-1/2}(\Gamma)$ into $L^2(\OOd)^3 \nxs\times\nxs L^2(\mathcal{G})$. Next, suppose that there exists $\ba \in \tilde{H}^{1/2}(\Gamma)^3 \nxs\times\nxs \tilde{H}^{1/2}(\Gamma) \nxs\times\nxs \tilde{H}^{-1/2}(\Gamma)$ such that $\mathcal{S^*}(\ba) = \bzero$. Then, the vanishing trace of $(\bu_{\ba}, p_{\ba})$ on $\mathcal{G}$ implies, by the unique continuation principle, that $(\bu_{\ba}, p_{\ba}) = \bzero$ in $\R^3 \backslash\nxs \overline{\Gamma}$, and whereby $(\dbr{\bu_{\ba}},\dbr{p_{\ba}},-\dbr{q_{\ba}}) = \ba = \bzero$ which guarantees the injectivity of $\mathcal{S}^*$. The densness of the range of~$\mathcal{S}^*$ is established via the injectivity of $\mathcal{S}$. To observe the latter, let $(\bg_s,g_f) \in L^2(\mathcal{G})^3\nxs\times\nxs L^2(\mathcal{G})$ such that $\mathcal{S}(\bg_s,g_f) = \bzero$. The definition~\eqref{oH} then reads $(\bn \exs\sip\exs \bC \colon\! \nabla \bu\ff,p\ff) = \bzero$ on $\Gamma$ with $(\bu\ff,p\ff)$ satisfying~\eqref{uf}. Now, assume that $\exs\Gamma$ contains $M\!\geqslant\!1$ (possibly disjoint) analytic surfaces~$\Gamma_m\!\subset\,\Gamma$, $m=1,\ldots M$, and consider the unique analytic continuation $\partial D_m$ of $\:\Gamma_m$ identifying an ``interior'' domain~$D_m\!\subset\mathcal{B}$. In this setting, $(\bn \exs\sip\exs \bC \colon\! \nabla \bu\ff,p\ff) = \bzero$ also on $\Gamma_m$ which thanks to the analyticity of $\bn \exs\sip\exs \bC \colon\! \nabla \bu\ff$ with respect to the surface coordinates on $\partial D_m$ leads to $(\bn \exs\sip\exs \bC \colon\! \nabla \bu\ff,p\ff) = \bzero$ on $\partial D_m$. Keeping in mind that $\omega \in \R$, observe that $\omega$ may not be an eigenfrequency associated with   
\beq\lb{uiH}
\begin{aligned}
&\nabla \exs\sip\exs (\bC \exs \colon \! \nabla \bu^{m}) \,-\, (\alpha-\dfrac{\rho_f}{\gamma}) \nabla p^{m}  \,+\, \omega^2 (\rho-\dfrac{\rho_f^2}{\gamma}) \bu^{m} \,=\, \bzero  \,\,    & \quad \text{in}\,\,\, &D_m, \\*[0.5mm] 
&\dfrac{1}{\gamma\omega^2} \nabla^2p^{m} \,+\, {M}^{-1} p^{m} \,+\, (\alpha-\dfrac{\rho_f}{\gamma}) \nabla \exs\sip\exs \bu^{m}  \,=\, \bzero  \,\,   & \quad \text{in}\,\,\, &D_m, \\*[1mm] 
&(\bn \exs\sip\exs \bC \colon\! \nabla \bu^{m},p^{m}) \,=\, \bzero      &\quad \textrm{on}~& \partial D_m, 
\end{aligned}
\eeq
which are inherently complex due to the complexity of $\gamma$, then $(\bu\ff,p\ff) = \bzero$ in $D_m$. The unique continuation principle then implies that $(\bu\ff,p\ff) = \bzero$ in $\R^3$ and thus $(\bg_s,g_f) = \bzero$ according to~\eqref{uf} which completes the proof. 

\section{Proof of Lemma~\ref{I{T}>0}}  \label{I{T}>0p}

The boundedness of $T$ stems from the well-posedness of the forward problem~\eqref{GE} and trace theorems. To observe~\eqref{pos-IT}, let $\bpsi \in H^{-1/2}(\Gamma)^3 \times H^{-1/2}(\Gamma) \times H^{1/2}(\Gamma)$ and consider $(\bu, p)$ satisfying \eqref{GE} with $(\bt\ff, q\ff,p\ff) = \bpsi$. Following similar steps used to obtain~\eqref{ImWik-GE}, one finds $\forall \exs \bpsi \in H^{-1/2}(\Gamma)^3 \times H^{-1/2}(\Gamma) \times H^{1/2}(\Gamma)$,
\beq\lb{T-bound}
\Im \dualGA{\bpsi}{T\bpsi} ~=~ -\frac{\omega}{\kappa} \nxs \norms{\nxs \bq \nxs}^2_{L^2(\R^3 \setminus\nxs\Gamma)} +~ \Im \exs \big\langle \mathfrak{P}\big(\dbr{\bu},\dbr{\exs {p} \xxs}, - \dbr{q} \big), \big(\dbr{\bu},\dbr{\exs {p} \xxs}, - \dbr{q} \big) \big\rangle_{\Gamma}.
\eeq
Then, the Lemma's claim follows immediately from~\eqref{ImP}.

\section{Proof of Lemma~\ref{T-invs0}}  \label{T-invs0p}
\begin{enumerate}
\item{}~First, observe that $T$ given by~\eqref{T} is Fredholm with index zero. The argument directly follows the proof of~\cite[Lemma 2.3]{Fiora2003} and makes use of the integral representation theorems and asymptotic behavior of the poroelastodynamics fundamental solution on $\Gamma$~\cite{scha2012}. One then further demonstrate the injectivity of~$T$ by assuming that there exists $\bpsi \in H^{-1/2}(\Gamma)^3 \times H^{-1/2}(\Gamma) \times H^{1/2}(\Gamma)$ so that $T(\bpsi)=\bzero$. Consequently,~\eqref{vinf2} reads that $(\bu, p) \!=\!\bzero$ in~$\mathbb{R}^3\nxs\backslash\nxs\overline{\Gamma}$, which implies that $(\bt, \llangle q \rrangle) \!=\!\bzero$ on $\Gamma$ according to~\eqref{Neu}. Then, the third to fifth of~\eqref{GE} indicate that $\bpsi = \bzero$, and thus the Lemma's claim follows.  
\item{}~We proceed using a contradiction argument. Assume that estimation \eqref{co-T0} does not hold true, then one can find a sequence $({\bpsi}_n)_{n\in\mathbb{N}^*}\subset H^{-1/2}(\Gamma)^3 \times H^{-1/2}(\Gamma) \times H^{1/2}(\Gamma)$ such that for all $n\in\mathbb{N}^*$,

\begin{equation}
\label{VarphiProperties}
 |\langle{\bpsi}_n,T({\bpsi}_n)\rangle| \,\,<\,\, \frac{1}{n}, \qquad \quad\|{\bpsi}_n\| = 1.
\end{equation}
Since $({\bpsi}_n)$ is bounded, one may assume up to extracting a subsequence that $({\bpsi}_n)$ converges weakly to some ${\bpsi}\in H^{-1/2}(\Gamma)^3 \times H^{-1/2}(\Gamma) \times H^{1/2}(\Gamma)$. Now, let $({\bu}_n,p_n) \in H^1_{\tiny\textrm{loc}}(\R^3 \nxs\setminus\nxs\overline{\Gamma})^3 \nxs\times\nxs H^1_{\tiny\textrm{loc}}(\R^3 \nxs\setminus\nxs\overline{\Gamma})$ solve the poroelastic scattering problem~\eqref{GE} for $(\bt\ff,q\ff,p\ff)={\bpsi}_n$. In this setting,~\eqref{T-bound} reads that for ${\bpsi}={\bpsi}_n$, 
\beq\lb{phinTphin}
 |\langle{\bpsi}_n,T({\bpsi}_n)\rangle| ~\geqslant~ \frac{\omega}{\kappa} \nxs \norms{\nxs \bq_n \nxs}^2_{L^2(\R^3\setminus\nxs\overline{\Gamma})^3} -~ \Im \exs \big\langle \mathfrak{P}\big(\dbr{\bu_n},\dbr{\exs {p_n} \xxs}, - \dbr{q_n} \big), \big(\dbr{\bu_n},\dbr{\exs {p_n} \xxs}, - \dbr{q_n} \big) \big\rangle_{\Gamma},
\eeq
wherein
\beq\lb{qnqn}
\begin{aligned}
&\bq_n ~:=~ \frac{1}{\gamma\omega^2}\big(\nabla p_n \,-\, \rho_f\omega^2\exs{\bu}_n\big) & \text{in} \,\,\,\, \R^3 \nxs\setminus\nxs\overline{\Gamma}, \\*[0.25mm]
& q_n ~:=~ \frac{1}{\gamma\omega^2}\big(\nabla p_n \sip\xxs \bn \,-\, \rho_f\omega^2\exs{\bu}_n \sip\xxs \bn\big) & \text{on} \,\,\,\, \Gamma.
\end{aligned}
\eeq

In light of~\eqref{ImP},~\eqref{VarphiProperties} and~\eqref{phinTphin}, observe that $\bq_n$ converges strongly to zero in $L^2(\mathbb{R}^3\nxs\setminus\nxs\overline{\Gamma})^3$. Next, from the wellposedness of the scattering problem and the continuity of the trace operator, one finds that  
\begin{equation}
\begin{aligned}
\big{\|}{\nxs\big(\dbr{\bu_n},\dbr{\exs {p_n} \xxs}, - \dbr{q_n} \big)\nxs}\big{\|}_{\tilde{H}^{1/2}(\Gamma)^3 \times \tilde{H}^{1/2}(\Gamma) \times \tilde{H}^{-1/2}(\Gamma)} ~&\leq\, \text{c} \, \|({\bu}_n,p_n)\|_{H^1_{\text{loc}}(\R^3\nxs\setminus\nxs\overline{\Gamma})^3 \times H^1_{\text{loc}}(\R^3\nxs\setminus\nxs\overline{\Gamma})}\\*[0.5mm]
&\hspace{-13mm}\leq\, \text{c}' \,  \|{\bpsi}_n\|_{H^{-1/2}(\Gamma)^3 \times H^{-1/2}(\Gamma) \times H^{1/2}(\Gamma)} ~=~ \text{c}',
\end{aligned}
\end{equation}
for a constant $\text{c}>0$ (\emph{resp.}~$\text{c}'>0$) independent of $({\bu}_n,p_n)$ (\emph{resp.}~${\bpsi}_n$). The boundedness of sequences $\dbr{\bu_n}$, $\dbr{\exs {p_n} \xxs}$ and $\dbr{q_n}$ indicates, up to extracting a subsequence, that they respectively converge weakly to some $\dbr{\bu}$, $\dbr{\exs {p} \xxs}$ and $\dbr{q}$. Then, the compactness of $\bar{\mathcal{S}}^*$, mapping $(\dbr{\bu_n},\dbr{\exs {p_n} \xxs}, - \dbr{q_n})$ to $({\bu}_n,p_n)$ according to~\eqref{vinf2}, implies that $({\bu}_n,p_n)$ converges strongly to some $(\bu,p)\in H^1_{\text{loc}}(\R^3\nxs\setminus\nxs\overline{\Gamma})^3 \times H^1_{\text{loc}}(\R^3\nxs\setminus\nxs\overline{\Gamma})$. 
From~\eqref{GE} and~\eqref{KS}, one may show that ${\bu}_n$ satisfies
\beq\lb{GEun}
\begin{aligned}
&\nabla \exs\sip\exs (\bC \exs \colon \! \nabla \bu_n) \,+\, (\rho - \alpha \rho_f) \exs \omega^2 \bu_n ~=~ (\alpha \gamma - \rho_f)\exs \omega^2 \bq_n \quad& \text{in}&\,\,\, {\R^3 \nxs\setminus\nxs \overline{\Gamma}}, \\*[0.5mm]
& \frac{\partial{\textrm{\bf u}}_\varsigma}{\partial r} \,-\, \text{i} k_\varsigma {\textrm{\bf u}}_\varsigma ~=~ o\big(r^{-1} e^{-\mathfrak{I}(k_\varsigma) r}\big),  \quad\,\,\,  \varsigma = {\textrm{p}_1},{\textrm{p}_2},{\textrm{s}}, \,\,\, \textrm{\bf u} = \bu_n \quad& \text{as}&\,\,\, r:=|\bxi|\to \infty.
\end{aligned}   
\eeq 

As $n \to \infty$, one obtains that $\bu$ satisfies
\beq\lb{GEu}
\begin{aligned}
&\nabla \exs\sip\exs (\bC \exs \colon \! \nabla \bu) \,+\, (\rho - \alpha \rho_f) \exs \omega^2 \bu ~=~ \bzero  \quad& \text{in}&\,\,\, {\R^3 \nxs\setminus\nxs \overline{\Gamma}}, \\*[0.5mm]
& \frac{\partial{\bu}_\varsigma}{\partial r} \,-\, \text{i} k_\varsigma {\bu}_\varsigma ~=~ o\big(r^{-1} e^{-\mathfrak{I}(k_\varsigma) r}\big),  \quad\,\,\,  \varsigma = {\textrm{p}_1},{\textrm{p}_2},{\textrm{s}},  & \text{as}&\,\,\, r:=|\bxi|\to \infty,
\end{aligned}   
\eeq 
implying that $\bu=\bzero$. Note that the conservative Navier system in first of~\eqref{GEu} does not generate the exponentially decaying radiation pattern in the second of~\eqref{GEu} (which is associated with the lossy Biot system). In this setting,~\eqref{qnqn} indicates that $p$ is constant in $\R^3 \nxs\setminus\nxs \overline{\Gamma}$ and since $p$ is subject to a radiation condition similar to the second of~\eqref{GEu}, one deduces that $p=0$. Now, from the wellposedness of the forward scattering problem, one finds that ${\bpsi}=0$. On the other hand, observe from~\eqref{frakP} that
\beq\lb{Nbpsi}
 \|{\bpsi}_n\nxs \|^2_{H^{-1/2}(\Gamma)^3 \times H^{-1/2}(\Gamma) \times H^{1/2}(\Gamma)} ~=~ \big\langle {\bpsi}_n, \mathfrak{P}(\dbr{\bu_n},\dbr{\xxs p_n \xxs},-\dbr{q_n \xxs}) \big\rangle_\Gamma \,-\, \big\langle {\bpsi}_n, (\bt_n,\llangle{q_n\nxs}\rrangle,\llangle{\xxs p_n \nxs}\rrangle) \big\rangle_\Gamma.
\eeq 
Then, the regularity of the trace operator implies that $(\dbr{\bu_n},\dbr{\xxs p_n \xxs},-\dbr{q_n \xxs})$ and $(\bt_n,\llangle{q_n\nxs}\rrangle,\llangle{\xxs p_n \nxs}\rrangle)$ converges strongly to zero. Consequently,~\eqref{Nbpsi} reads that $\|{\bpsi}_n \|^2$ goes to zero as $n$ goes to infinity. This result contradicts~\eqref{VarphiProperties} and establishes the lemma's statement.
\end{enumerate}

\section{Proof of Lemma~\ref{FF_op}}  \label{FF_opp}

The injectivity of $\Lambda$ is implied by the injectivity of operators ${\mathcal{S}}^*$, $T$, and $\mathcal{S}$ in the factorization
$\Lambda = \bar{\mathcal{S}}^* \exs T \exs \mathcal{S}$ (see Lemmas~\ref{H*p} and~\ref{T-invs0}). The compactness of~$\Lambda$ follows immediately from the compactness of  $\mathcal{S}^*$ -- and thus that of $\mathcal{S}$ (Lemma~\ref{H*p}), and the boundedness of~$T$ (Lemma~\ref{I{T}>0}). The last claim may be verified by establishing the injectivity of~$\Lambda^* = \mathcal{S}^* \mathcal{P}^*$ which in light of Lemma~\ref{H*p}, results from the injectivity of $\mathcal{P}^*$. To demonstrate the latter, one first finds from the definition of $\mathcal{P}$ that the adjoint operator $\mathcal{P}^*\colon\nxs L^2(\OOd)^3\times L^2(\mathcal{G}) \rightarrow \tilde{H}^{1/2}(\Gamma)^3 \nxs\times\nxs \tilde{H}^{1/2}(\Gamma) \nxs\times\nxs \tilde{H}^{-1/2}(\Gamma)$ is given by $\mathcal{P}^*(\bg) = (\dbr{\xxs \bu^\star},\dbr{\exs {p}^\star \xxs}, - \dbr{\xxs q^\star} )$ on $\Gamma$ where $(\bu^\star, p^\star) \in H^1_{\tiny\textrm{loc}}(\R^3 \nxs\setminus\nxs \lbrace\overline{\Gamma} \cup \overline{\mathcal{G}}\rbrace)^3 \nxs\times\nxs H^1_{\tiny\textrm{loc}}(\R^3 \nxs\setminus\nxs \lbrace\overline{\Gamma} \cup \overline{\mathcal{G}}\rbrace)$ solves 
\beq\lb{ustar}
\begin{aligned}
&\nabla \exs\sip\exs (\bC \exs \colon \! \nabla \bu^\star)(\bxi) \,-\, (\alpha-\dfrac{\rho_f}{\bar\gamma}) \nabla p^\star(\bxi)  \,+\, \omega^2 (\rho-\dfrac{\rho_f^2}{\bar\gamma}) \bu^\star(\bxi)\,=\, -\int_\mathcal{G} \delta(\by-\bxi) \bg_s(\by) \exs \textrm{d}S_{\by} \quad& \text{in}&\,\,\, {\R^3 \nxs\setminus\nxs \lbrace\overline{\Gamma} \cup \overline{\mathcal{G}}\rbrace}, \\*[1mm]
&\dfrac{1}{\bar{\gamma}\omega^2} \nabla^2p^\star(\bxi) \,+\, {M}^{-1} p^\star(\bxi) \,+\, (\alpha-\dfrac{\rho_f}{\bar\gamma}) \nabla \exs\sip\exs \bu^\star(\bxi)  \,=\, -   \int_\mathcal{G} \delta(\by-\bxi) g_f(\by) \exs \textrm{d}S_{\by}   \quad& \text{in}&\,\,\, {\R^3 \nxs\setminus\nxs \lbrace\overline{\Gamma} \cup \overline{\mathcal{G}}\rbrace}, \\*[2.5mm]
& (\bt^\star, \llangle\exs q^\star \rrangle, \llangle\exs p^\star \rrangle) ~=~ \bar{\mathfrak{P}}^{\textrm{T}}(\dbr{\xxs \bu^\star},\dbr{\exs {p}^\star \xxs}, - \dbr{\xxs q^\star} ), \quad \dbr{\xxs \bt^\star} ~=~ \bzero  \quad&  \text{on}&\,\,\, {\Gamma}, 
\end{aligned}   
\eeq  
wherein `T' indicates transpose, and
\[
[\bt^\star,q^\star](\bu^\star,p^\star) \,:=\, \big[\bn \nxs\cdot\nxs \bC \colon \!\nxs \nabla \bu^\star - \alpha \exs p^\star  \bn, \, \dfrac{1}{\bar{\gamma}\omega^2}(\nabla p^\star \sip\exs \bn - \rho_f \omega^2 \bu^\star \sip\exs \bn) \big] \qquad  \text{on} \,\,\,  {\Gamma}. 
\]
Note that $(\bu^\star,p^\star)$ satisfy the radiation condition as $|\bxi | \rightarrow \infty$, similar to the complex conjugate of~\eqref{KS}. It should be mentioned that $\mathcal{P}^*$ is obtained by applying the poroelastodynamics reciprocal theorem between $(\bu^\star, p^\star)$ of~\eqref{ustar} and $(\bu, p)$ satisfying~\eqref{GE} similar to~\ref{H*pruf}. Now, let $\bg \in L^2(\OOd)^3 \times L^2(\mathcal{G})$, and assume that $\mathcal{P}^*(\bg) = \bzero$, i.e.,~$(\dbr{\xxs \bu^\star},\dbr{\exs {p}^\star \xxs}, - \dbr{\xxs q^\star} ) = \bzero$.
In this case, $(\bu^\star,p^\star)$ and $q^\star$ are continuous across $\Gamma$ independent of $\bar{\mathfrak{P}}^{\textrm{T}}$. Therefore, $(\bu^\star, p^\star) \in H^1_{\tiny\textrm{loc}}(\R^3 \nxs\setminus\nxs\overline{\mathcal{G}})^3 \nxs\times\nxs H^1_{\tiny\textrm{loc}}(\R^3\nxs\setminus\nxs\overline{\mathcal{G}})$ and $(\bt^\star,  q^\star ,  p^\star ) = \bzero$. Thanks to this result, one can show as in the proof of Lemma~\ref{H*p} that $(\bu^\star,p^\star) = \bzero$, and consequently $\bg = \bzero$ which concludes the proof.

%
\bibliography{inverse,crackbib}

\begin{thebibliography}{10}
\expandafter\ifx\csname url\endcsname\relax
  \def\url#1{\texttt{#1}}\fi
\expandafter\ifx\csname urlprefix\endcsname\relax\def\urlprefix{URL }\fi
\expandafter\ifx\csname href\endcsname\relax
  \def\href#1#2{#2} \def\path#1{#1}\fi

\bibitem{hofm2019}
H.~Hofmann, G.~Zimmermann, M.~Farkas, E.~Huenges, A.~Zang, M.~Leonhardt,
  G.~Kwiatek, P.~Martinez-Garzon, M.~Bohnhoff, K.-B. Min, P.~Fokker,
  R.~Westaway, F.~Bethmann, P.~Meier, K.~S. Yoon, J.~W. Choi, T.~J. Lee, K.~Y.
  Kim, First field application of cyclic soft stimulation at the pohang
  enhanced geothermal system site in korea, Geophysical Journal International
  217~(2) (2019) 926--949.

\bibitem{cau2017}
R.~A. Caulk, I.~Tomac, Reuse of abandoned oil and gas wells for geothermal
  energy production, Renewable Energy 112 (2017) 388--397.

\bibitem{mass2017}
Y.~Masson, B.~Romanowicz, Box tomography: Localized imaging of remote targets
  buried in an unknown medium, a step forward for understanding key structures
  in the deep earth, Geophysical Journal International (2017) ggx141.

\bibitem{Maxw2014}
S.~Maxwell, Microseismic imaging of hydraulic fracturing: Improved engineering
  of unconventional shale reservoirs, Society of Exploration Geophysicists,
  2014.

\bibitem{hou2021}
L.~Hou, D.~Elsworth, Mechanisms of tripartite permeability evolution for
  supercritical co2 in propped shale fractures, Fuel 292 (2021) 120188.

\bibitem{snee2020}
J.-E.~L. Snee, M.~D. Zoback, Multiscale variations of the crustal stress field
  throughout north america, Nature communications 11~(1) (2020) 1--9.

\bibitem{zoba2019}
M.~D. Zoback, A.~H. Kohli, Unconventional reservoir geomechanics, Cambridge
  University Press, 2019.

\bibitem{Geoe2017}
N.~Watanabe, G.~Blocher, M.~Cacace, S.~Held, T.~Kohl, Geoenergy modeling
  III~enhanced geothermal systems, Springer, 2017.

\bibitem{wape2021}
K.~Wapenaar, R.~Snieder, S.~de~Ridder, E.~Slob, Green's function
  representations for marchenko imaging without up/down decomposition, arXiv
  preprint arXiv:2103.07734 (2021).

\bibitem{sun2020}
H.~Sun, L.~Demanet, Extrapolated full-waveform inversion with deep learning,
  Geophysics 85~(3) (2020) R275--R288.

\bibitem{meti2018}
L.~M{\'e}tivier, A.~Allain, R.~Brossier, Q.~M{\'e}rigot, E.~Oudet, J.~Virieux,
  Optimal transport for mitigating cycle skipping in full-waveform inversion: A
  graph-space transform approach, Geophysics 83~(5) (2018) R515--R540.

\bibitem{verd2020}
J.~P. Verdon, S.~A. Horne, A.~Clarke, A.~L. Stork, A.~F. Baird, J.-M. Kendall,
  Microseismic monitoring using a fiber-optic distributed acoustic sensor
  array, Geophysics 85~(3) (2020) KS89--KS99.

\bibitem{Resh2010}
A.~Reshetnikov, S.~Buske, S.~A. Shapiro, Seismic imaging using microseismic
  events:~results from the san andreas fault system at safod, Journal of
  Geophysical Research:~Solid Earth 115~(B12) (2010) B12324.

\bibitem{calo2011}
M.~Cal{\`o}, C.~Dorbath, F.~H. Cornet, N.~Cuenot, Large-scale aseismic motion
  identified through 4-dp-wave tomography, Geophysical Journal International
  186~(3) (2011) 1295--1314.

\bibitem{Gaje2017}
W.~Gajek, J.~Verdon, M.~Malinowski, J.~Trojanowski, Imaging seismic anisotropy
  in a shale gas reservoir by combining microseismic and 3d surface reflection
  seismic data, in: 79th EAGE Conference \& Exhibition, Paris, France, 2017.

\bibitem{shap2015}
S.~A. Shapiro, Fluid-induced seismicity, Cambridge University Press, 2015.

\bibitem{baig2010}
A.~Baig, T.~Urbancic, Microseismic moment tensors: A path to understanding frac
  growth, The Leading Edge 29~(3) (2010) 320--324.

\bibitem{teod2021}
D.~Teodor, C.~Cesare, F.~Khosro~Anjom, R.~Brossier, V.~S. Laura, J.~Virieux,
  Challenges in shallow targets reconstruction by 3d elastic full-waveform
  inversion: Which initial model?, Geophysics 86~(4) (2021) 1--91.

\bibitem{trin2019}
P.-T. Trinh, R.~Brossier, L.~M{\'e}tivier, L.~Tavard, J.~Virieux, Efficient
  time-domain 3d elastic and viscoelastic full-waveform inversion using a
  spectral-element method on flexible cartesian-based mesh, Geophysics 84~(1)
  (2019) R75--R97.

\bibitem{Audibert2014}
L.~Audibert, H.~Haddar, A generalized formulation of the linear sampling method
  with exact characterization of targets in terms of farfield measurements,
  Inverse Problems 30 (2014) 035011.

\bibitem{audi2017}
L.~Audibert, H.~Haddar, The generalized linear sampling method for limited
  aperture measurements, SIAM Journal on Imaging Sciences 10~(2) (2017)
  845--870.

\bibitem{bonn2019}
M.~Bonnet, F.~Cakoni, Analysis of topological derivative as a tool for
  qualitative identification, Inverse Problems 35~(10) (2019) 104007.

\bibitem{pour2019}
F.~Pourahmadian, H.~Haddar, Differential tomography of micromechanical
  evolution in elastic materials of unknown micro/macrostructure, SIAM Journal
  on Imaging Sciences 13~(3) (2020) 1302--1330.

\bibitem{cako2016}
F.~Cakoni, D.~Colton, H.~Haddar, Inverse Scattering Theory and Transmission
  Eigenvalues, SIAM, 2016.

\bibitem{pour2020}
F.~Pourahmadian, Experimental validation of differential evolution indicators
  for ultrasonic waveform tomography, arxiv preprint arXiv:2010.01813 (2020).

\bibitem{pour2020(2)}
F.~Pourahmadian, H.~Yue, Laboratory application of sampling approaches to
  inverse scattering, Inverse Problems 37~(5) (2021) 055012.

\bibitem{rubi2014}
J.~G. Rubino, T.~M. M{\"u}ller, L.~Guarracino, M.~Milani, K.~Holliger,
  Seismoacoustic signatures of fracture connectivity, Journal of Geophysical
  Research: Solid Earth 119~(3) (2014) 2252--2271.

\bibitem{germ2013}
J.~G. Rubino, L.~Guarracino, T.~M. M{\"u}ller, K.~Holliger, Do seismic waves
  sense fracture connectivity?, Geophysical Research Letters 40~(4) (2013)
  692--696.

\bibitem{alla2009}
J.~Allard, N.~Atalla, Propagation of sound in porous media: modeling sound
  absorbing materials, John Wiley \& Sons, 2009.

\bibitem{qiao2017}
X.~Qiao, Z.~Shao, W.~Bao, Q.~Rong, Fiber bragg grating sensors for the oil
  industry, Sensors 17~(3) (2017) 429.

\bibitem{biot1956}
M.~A. Biot, Theory of propagation of elastic waves in a fluid-saturated porous
  solid. ii.~higher frequency range, The Journal of the acoustical Society of
  America 28~(2) (1956) 179--191.

\bibitem{biot1962}
M.~A. Biot, Mechanics of deformation and acoustic propagation in porous media,
  Journal of applied physics 33~(4) (1962) 1482--1498.

\bibitem{hyma2016}
J.~D. Hyman, J.~Jim{\'e}nez-Mart{\'\i}nez, H.~S. Viswanathan, J.~W. Carey,
  M.~L. Porter, E.~Rougier, S.~Karra, Q.~Kang, L.~Frash, L.~Chen, Z.~Lei,
  D.~O'Malley, N.~Makedonska, Understanding hydraulic fracturing: a multi-scale
  problem, Philos Trans A Math Phys Eng Sci 374~(2078) (2016) 20150426.

\bibitem{Scaling2003}
G.~I. Barenblatt, Scaling (Cambridge texts in applied mathematics), Cambridge
  University Press, Cambridge, UK, 2003.

\bibitem{naka2007}
S.~Nakagawa, M.~A. Schoenberg, Poroelastic modeling of seismic boundary
  conditions across a fracture, The Journal of the Acoustical Society of
  America 122~(2) (2007) 831--847.

\bibitem{norr1985}
A.~N. Norris, Radiation from a point source and scattering theory in a
  fluid-saturated porous solid, The Journal of the Acoustical Society of
  America 77~(6) (1985) 2012--2023.

\bibitem{bour1992}
T.~Bourbi{\'e}, O.~Coussy, B.~Zinszner, M.~C. Junger, Acoustics of porous
  media, Acoustical Society of America, 1992.

\bibitem{McLean2000}
W.~McLean, Strongly Elliptic Systems and Boundary Integral Equations, Cambridge
  University Press, Cambridge, 2000.

\bibitem{chen1991}
A.~H.-D. Cheng, T.~Badmus, D.~E. Beskos, Integral equation for dynamic
  poroelasticity in frequency domain with bem solution, Journal of engineering
  mechanics 117~(5) (1991) 1136--1157.

\bibitem{scha2012}
M.~Schanz, Wave propagation in viscoelastic and poroelastic continua: a
  boundary element approach, Vol.~2, Springer Science \& Business Media, 2012.

\bibitem{pour2017}
F.~Pourahmadian, B.~B. Guzina, H.~Haddar, Generalized linear sampling method
  for elastic-wave sensing of heterogeneous fractures, Inverse Problems 33~(5)
  (2017) 055007.

\bibitem{Kress1999}
R.~Kress, Linear integral equation, Springer, Berlin, 1999.

\bibitem{Kirsch2008}
A.~Kirsch, N.~Grinberg, The factorization methods for inverse problems, Oxford
  University Press, Oxford, 2008.

\bibitem{pour2017(2)}
F.~Pourahmadian, B.~B. Guzina, H.~Haddar, A synoptic approach to the seismic
  sensing of heterogeneous fractures: from geometric reconstruction to
  interfacial characterization, Computer Methods in Applied Mechanics and
  Engineering 324 (2017) 395--412.

\bibitem{pour2018}
F.~Pourahmadian, B.~B. Guzina, On the elastic anatomy of heterogeneous
  fractures in rock, International Journal of Rock Mechanics and Mining
  Sciences 106 (2018) 259--268.

\bibitem{Hech2012}
F.~Hecht, \href{https://freefem.org/}{New development in freefem++}, Journal of
  Numerical Mathematics 20~(3-4) (2012) 251--265.
\newline\urlprefix\url{https://freefem.org/}

\bibitem{ding2013}
B.~Ding, A.~H.-D. Cheng, Z.~Chen, Fundamental solutions of poroelastodynamics
  in frequency domain based on wave decomposition, Journal of Applied Mechanics
  80~(6) (2013).

\bibitem{yew1976}
C.~H. Yew, P.~N. Jogi, Study of wave motions in fluid-saturated porous rocks,
  The Journal of the Acoustical Society of America 60~(1) (1976) 2--8.

\bibitem{nguy2019}
T.-P. Nguyen, B.~B. Guzina, Generalized linear sampling method for the inverse
  elastic scattering of fractures in finite bodies, Inverse Problems 35~(10)
  (2019) 104002.

\bibitem{chen2016}
A.~H.-D. Cheng, Poroelasticity, Vol.~27, Springer, 2016.

\bibitem{Fiora2003}
F.~Cakoni, D.~Colton, The linear sampling method for cracks, Inverse Problems.
  19 (2003) 279--295.

\end{thebibliography}

\end{document}